\documentclass[letterpaper,12pt]{extarticle}
\usepackage[margin=1.4in,letterpaper]{geometry} 
\usepackage[dvips]{graphicx}
\usepackage{fancyhdr}
\usepackage{color}
\usepackage{theoremref, hyperref}
\usepackage{amsmath, amsthm}
\usepackage{amssymb}
\usepackage{tikz}
\usepackage{bm}
\usepackage{url}
\usepackage{latexsym}
\usepackage{arydshln}

 \numberwithin{equation}{section}

\definecolor{darkcyan}{rgb}{0.0, 0.55, 0.55}

\newcommand{\Hil}[0]{
\mathcal{H}
}

\newcommand{\Rd}[0]{
{\mathbb{R}^d}
}

\newcommand{\B}[0]{{\mathcal{B}}}

\newcommand{\R}[0]{{\mathcal{R}}}
\newcommand{\G}[0]{{\mathcal{G}}}
\newcommand{\N}[0]{{\mathcal{N}}}
\newcommand{\A}[0]{{\mathcal{A}}}
\newcommand{\J}[0]{{\mathcal{J}}}
\newcommand{\MS}[0]{{\mathcal{S}}}
\newcommand{\C}[0]{{\mathcal{C}}}
\newcommand{\K}[0]{{\mathcal{K}}}

\newtheorem{theorem}{Theorem}[section]
\newtheorem{definition}[theorem]{Definition}
\newtheorem{proposition}[theorem]{Proposition}
\newtheorem{lemma}[theorem]{Lemma}
\newtheorem{corollary}[theorem]{Corollary}
\newtheorem{remark}[theorem]{Remark}
\newtheorem{ex}[theorem]{Example}
\newtheorem{conj.}[theorem]{Conjecture}
\newtheorem{Bsp.}[theorem]{Example}

\begin{document}
%
%
\title{\bf\vspace{-10
pt} 
Localization of operator-valued frames
}
%
%
\author{L. Köhldorfer and P. Balazs}

%
%
\date{}

%
%
\maketitle

\begin{abstract}
{We introduce a localization concept for operator-valued frames, where the quality of localization is measured by the associated operator-valued Gram matrix belonging to some suitable Banach algebra. We prove that intrinsic localization of an operator-valued frame is preserved by its canonical dual. Moreover, we show that the series associated to the perfect reconstruction of an operator-valued frame converges not only in the underlying Hilbert space, but also in a whole class of associated (quasi-)Banach spaces. Finally, we apply our results to irregular Gabor g-frames.}
\end{abstract}

\section{Introduction}

Understanding complicated objects using simpler and easier-to-handle building blocks is an old and established concept in mathematics. In Hilbert space the most common systems of building blocks are orthonormal bases, which allow a unique representation of any vector. However, both in theory and in practice, more flexible systems are often required to analyze vectors. One particularly well-established generalization of orthonormal bases are \emph{frames} \cite{duffschaef1,daubgromay86,ole1n}, which allow redundant and stable representations of vectors in a Hilbert space. The study of frames led to many deep theoretical results, such as versions of the \emph{Balian-Low Theorem} \cite{HeWe_96,BeHeWal_95}, or the \emph{Feichtinger Conjecture}, which was proven by Marcus, Spielmann and Srivastava \cite{ MarSpiSri_15} and is equivalent to many other deep theorems in functional analysis \cite{CasChrLinVer_05}. On the other hand, the abstract properties of frames are also desirable for a variety of applications, such as signal processing \cite{framepsycho16}, compressed sensing \cite{rascva06} and many more, see \cite{ole1n} and the numerous references therein. 

Here, we work in a setting which is an extension of the latter. Our work is on operator-valued frames, because, despite its great usefulness in many situations, frame theory also has its limitations. One can think of a frame as a weighted family of rank-one projections. However, in various applications, such as distributed- and parallel processing \cite{cakuli08,CK3}, wireless sensor networks \cite{IB}, packet encoding \cite{B,BKP,CKLR}, visual and hearing systems \cite{dist3}, or geophones in geophysics \cite{CG}, the linear reconstruction of vectors from \emph{higher-rank measurements} is required. This lead to the study of \emph{fusion frames} \cite{caskut04,cakuli08,koebacasheihomosha23} and further generalizations of frames such as \emph{operator-valued frames}. 

The theory of operator-valued frames (originally called \emph{generalized frames}, or simply \emph{g-frames}) was introduced by W. Sun in \cite{Sun2006437} and encompasses not only frames and fusion frames, but also other established mathematical notions such as \emph{bounded quasi-projectors} \cite{forn03,forn04}, or \emph{time-frequency localization operators} \cite{döfeigrö06}. A countable family $T=(T_k)_{k\in X}$ of bounded operators $T_k\in \B(\Hil)$ on some separable Hilbert space $\Hil$ is called an operator-valued frame, if there exist positive constants $0<A\leq B < \infty$, such that
\begin{equation}\label{gframeineq}
A \Vert f \Vert^2 \leq \sum_{k\in X} \Vert T_k f \Vert ^2 \leq B \Vert f \Vert^2 \qquad (\forall f\in \Hil).
\end{equation}
Abstract theory shows \cite{Sun2006437}, that every operator-valued frame $(T_k)_{k\in X}$ allows for a linear, bounded and stable reconstruction of any vector $f\in \Hil$ from the $\Hil$-valued sequence $(T_k f)_{k\in X} \in \ell^2(X;\Hil)$. 
In addition to the possibility of perfect reconstruction, 
operator-valued frames share also many other parallels to frames, and their properties are well-understood \cite{Sun2006437,SUN2007858,kutpatphi17}.

However, beside the latter mentioned features, operator-valued frames might admit other useful properties, which cannot be captured by the defining inequalities (\ref{gframeineq}) alone. For instance, a natural constraint for an operator-valued frame $(T_k)_{k\in X}$ would be the requirement, that it is \emph{well localized} in the sense that different elements $T_k$ essentially act on different parts of the space $\Hil$ and do not correlate much with each other. For example, we could consider an operator-valued frame $(T_k)_{k\in X}$, which additionally satisfies $\ker(T_l)^{\perp} \subseteq \ker(T_k)$ whenever $k\neq l$, or, equivalently, 
\begin{equation}\label{toostrict}
\Vert T_k T_l^*\Vert = 0 \qquad \text{whenever } k\neq l.  
\end{equation}
Although such operator-valued frames do exist (e.g. every orthonormal fusion basis \cite{cakuli08,koebacasheihomosha23} has this property), the latter condition will only be satisfied for a very small class of operator-valued frames. Hence, it seems more reasonable to study a broader class of operator-valued frames satisfying a relaxed version of (\ref{toostrict}), which will be done here.  

Our main motivation for writing this article comes from considering \emph{Gabor g-frames} for $L^2(\Rd)$ \cite{skret2020}. These are operator-valued frames $(T_k)_{k\in X}$ consisting of \emph{translations} of a fixed \emph{window operator} in the time-frequency plane, generalizing Gabor frames in $L^2(\Rd)$. We show in Theorem \ref{polynomiallylocalizedgframe} that if the window operator is suitably chosen and if the index set $X\subset\mathbb{R}^{2d}$ is sufficiently regular, then there exists some uniform constant $C>0$ such that 
\begin{equation}\label{polyintro}
    \Vert T_k T_l^*\Vert \leq C (1+\vert k-l\vert)^{-s} \qquad (\forall k,l\in X),
\end{equation}
where, similar to the standard Gabor setting,  the decay parameter $s>2d$ stems from the operator class to which $T$ belongs. The latter estimate can be interpreted to mean that $T_k$ and $T_l$ correlate very little, whenever $k$ and $l$ are far apart. In particular, since $T_k T_l^*$ is precisely the $(k,l)$-th entry of the $\B(L^2(\Rd))$-valued Gram matrix $G_T$ associated with the family $(T_k)_{k\in X}$, we may reinterpret (\ref{polyintro}) to saying that the Gram matrix $G_T$ has polynomial off-diagonal decay. In more abstract terms, this means that $G_T$ belongs to the \emph{Jaffard class} of operator-valued matrices, a matrix algebra with manifold convenient properties, that has recently been studied in 
\cite{koeba25,KoeBaSAMPTA25}.
This observation suggests to define localization of an operator-valued frame as proposed in \cite[Definition 5.7]{Krishtal2011} and analogously done for frames in \cite{forngroech1} by declaring an operator-valued frame $T=(T_k)_{k\in X}$ \emph{intrinsically $\A$-localized}, if its associated Gram matrix $G_T$ belongs to a given \emph{solid spectral matrix algebra} $\A$ (see Definition \ref{spectralalgebra}).

In this article, we provide a systematic study of localized operator-valued frames as defined above. We show that intrinsic localization is preserved under canonical duality. For every mutually localized pair of operator-valued dual frames we define a whole chain of associated co-orbit spaces. We prove that the localization assumption of an operator-valued frame guarantees that all of these spaces are well-defined (quasi-)Banach spaces and show that the series associated to perfect reconstruction of such operator-valued frames not only converge in the underlying Hilbert space, but also in each of these spaces. Finally, we apply our results to (possibly irregular) Gabor g-frames for $L^2(\mathbb{R^d})$.

\section{Preliminaries and Notation}\label{Preliminaries and Notation}

We denote the cardinality of a set $X$ by $\vert X \vert$. For $x\in \mathbb{R}^d$, $\vert x\vert$ denotes the Euclidean norm on $\mathbb{R}^d$. The symbol $\mathbb{N}$ denotes the set of positive integers $\lbrace 1, 2, 3, \dots \rbrace$. 
The kernel and range of an operator $T$ is denoted by $\mathcal{N}(T)$ and  $\mathcal{R}(T)$  respectively, and  $\mathcal{I}_X$ denotes the identity element on a given space $X$. The space of bounded operators between two (quasi-)normed spaces $X$ and $Y$ is denoted by $\mathcal{B}(X,Y)$\index{$\mathcal{B}(X,Y)$} and we set $\mathcal{B}(X) := \mathcal{B}(X,X)$. We write $X \simeq Y$ if two (quasi-)normed spaces $X$ and $Y$ are isomorphic and $X \cong Y$ if they are isometrically isomorphic. The topological dual space of a normed space $X$ is denoted by $X^{*}$. 

Throughout these notes, $\Hil$ is a separable Hilbert space and $\Vert \cdot \Vert$ (without an index) always denotes either the norm in $\Hil$ or the norm in $\B(\Hil)$. All other occurring norms will be labeled with an index. 

\subsection{The Moore-Penrose pseudo-inverse}\label{The Moore-Penrose pseudo-inverse}

Recall the construction of the (Moore-Penrose) pseudo-inverse of a bounded closed range operator between two Hilbert spaces \cite{ole1n}: For $U \in \mathcal{B}( \mathcal{H}_1 , \mathcal{H}_2 )$ having closed range, the restriction $U_0 := U\vert_{\mathcal{N}(U)^{\perp}} : \mathcal{N}(U)^{\perp} \longrightarrow \mathcal{H}_2$, is bounded and injective with $\mathcal{R}(U_0) = \mathcal{R}(U)$. Hence, $U_0^{-1} :\mathcal{R}(U) \longrightarrow \mathcal{N}(U)^{\perp}$ is well-defined and bounded, and the extension $U^{\dagger}$ of $U_0^{-1}$, defined by 
$$U^{\dagger} x = \begin{cases}
U_0^{-1} x & \text{  if } x\in \mathcal{R}(U)\\
0 & \text{  if } x\in \mathcal{R}(U)^{\perp}
\end{cases}$$
is called the \emph{(Moore-Penrose) pseudo inverse} of $U$. In case $U$ is surjective we have $U^{\dagger} = U^*(UU^*)^{-1}$, and if $U$ is even bijective then $U^{\dagger} = U^{-1}$.

The pseudo-inverse can be characterized as follows.

\begin{lemma}\label{pseudoinversechar} \cite{BEUTLER1976397}
    Let $U \in \mathcal{B}( \mathcal{H}_1 , \mathcal{H}_2 )$ be a bounded operator between Hilbert spaces and assume that $U$ has closed range. The pseudo-inverse of $U$ is the unique operator $U^{\dagger}: \mathcal{H}_2 \longrightarrow  \mathcal{H}_1$, satisfying the three relations 
\begin{equation}\label{relationspseudo}
\mathcal{N}(U^{\dagger}) = \mathcal{R}(U)^{\perp}, \qquad \mathcal{R}(U^{\dagger}) = \mathcal{N}(U)^{\perp}, \qquad UU^{\dagger}U = U.
\end{equation}
\end{lemma}

The following formula for the pseudo-inverse is folklore and can $-$ for instance $-$ be deduced from \cite[Lemma 2.5.2 (iv)]{ole1n} or \cite[Corollary 2.3]{olepinv} together with the relations (\ref{relationspseudo}).  

\begin{lemma}\label{pseudoinverseformula}
Let $U\in \mathcal{B}( \mathcal{H}_1 , \mathcal{H}_2 )$ be a bounded operator between Hilbert spaces and assume that $U$ has closed range. Then
$$U^{\dagger} = U^* (UU^*)^{\dagger} .$$ 
\end{lemma}

\subsection{Weight functions}\label{Weight functions}

A \emph{weight function} on $\mathbb{R}^{d}$, or simply \emph{weight}, is a continuous and positive function $\nu :\mathbb{R}^{d} \longrightarrow (0,\infty )$. A weight $\nu$ is called \emph{submultiplicative}, if 
$$\nu (x+x') \leq \nu(x)\nu(x') \qquad (\forall x, x' \in \mathbb{R}^{d}),$$
and \emph{symmetric} if 
$$\nu(x) = \nu(-x) \qquad (\forall x \in \mathbb{R}^d).$$
A weight $m$ is called \emph{$\nu$-moderate} if there exists some constant $C>0$ such that 
$$m(x+x') \leq Cm(x)\nu(x') \qquad (\forall x, x' \in \mathbb{R}^{d}).$$
We say that $\nu$ satisfies the \emph{GRS-condition} (Gelfand-Raikov-Shilov condition \cite{GRS64}), if
$$\lim_{n\rightarrow \infty} (\nu(nz))^{\frac{1}{n}} = 1 \qquad (\forall z\in \mathbb{R}^{d}).$$
For more background on weight functions see \cite{gr06weightsinTFA}.

\subsection{Bochner sequence spaces}

Occasionally, our analysis will rely on properties on Banach space-valued $\ell^p$-spaces, defined as follows. If $B$ is a Banach space, $X$ a countable index set, and $\nu = (\nu_k)_{k\in X}$ a family of positive numbers (e.g. samples of a weight), then for each $p\in (0,\infty]$ the \emph{Bochner sequence space} $\ell^p(X; B)$ \cite[Chapter 1]{HyNeVeWe16} is defined by 
$$\ell_{\nu}^p(X; B) := \left\lbrace (Y_k)_{k\in X}: Y_k \in B \, (\forall k\in X), (\Vert Y_k \Vert_B \cdot \nu_k)_{k\in X} \in \ell^p(X) \right\rbrace .$$
In case $\nu_k = 1$ for all $k\in X$, we write $\ell_{\nu}^p(X; B) = \ell^p(X; B)$.

The Bochner sequence spaces $\ell_{\nu}^p(X; B)$ share many properties with the classical $\ell^p$-spaces. In particular (see \cite[Chapter 1]{HyNeVeWe16}), $\ell_{\nu}^p(X; B)$ equipped with the (quasi-)norm 
$$\Vert (Y_k)_{k\in X} \Vert_{\ell_{\nu}^p(X; B)} = \Vert (\Vert Y_k \Vert_B \cdot \nu_k)_{k\in X}\Vert_{\ell^p(X)},$$
is a Banach space for every $1\leq p \leq \infty$ and a quasi-Banach space for $0<p<1$. If $1\leq p <\infty$ and $\frac{1}{p}+\frac{1}{q} = 1$, then the dual space $(\ell_{\nu}^p(X; B))^*$ of $\ell_{\nu}^p(X; B)$ is isometrically isomorphic to $\ell_{1/\nu}^q(X; B^*)$. Hence, we identify the latter with $(\ell_{\nu}^p(X; B))^*$ from now on. 

By $\ell^{00}(X;B)$ we denote the space of all finitely supported $B$-valued sequences indexed by $X$, and by $\ell_{\nu}^{0}(X;B)$ we denote the space of all $B$-valued sequences $(f_k)_{k\in X}$ such that for all $\varepsilon > 0$ there exists $(g_k)_{k\in X} \in \ell^{00}(X;B)$ satisfying $\sup_{k\in X} \Vert f_k - g_k \Vert \nu(l) < \varepsilon$. It is straightforward to see that $(\ell^{00}(X;B), \Vert \, . \, \Vert_{\ell_{\nu}^p(X; B)})$ is a dense subspace of $(\ell_{\nu}^p(X;B), \Vert \, . \, \Vert_{\ell_{\nu}^p(X; B)})$ for each $p\in [1,\infty)$. By definition, $(\ell_{\nu}^{00}(X;B), \Vert \, . \, \Vert_{\ell_{\nu}^{0}(X; B)})$ is also a dense subspace of $(\ell_{\nu}^{0}(X;B), \Vert \, . \, \Vert_{\ell_{\nu}^{0}(X; B)})$, where we set $\Vert \, . \, \Vert_{\ell_{\nu}^{0}(X; B)} := \Vert \, . \, \Vert_{\ell_{\nu}^{\infty}(X; B)}$ for notational convenience. Moreover, in \cite[Theorem 3.1]{YaSr14}  it was proven that $(\ell_{\nu}^{0}(X;B), \Vert \, . \, \Vert_{\ell_{\nu}^{0}(X; B)})^*$ is isomorphic to $(\ell_{1/\nu}^1(X;B^*), \Vert \, . \, \Vert_{\ell_{1/\nu}^1(X; B^*)})$.

An important property of Bochner sequence spaces is the fact that Riesz-Thorin interpolation holds. For later reference, we state the following special case of \cite[Theorem 2.2.1]{HyNeVeWe16}.

\begin{proposition}[Riesz-Thorin interpolation]\label{Riesz-Thorin}
Assume that $A\in \mathcal{B}(\ell_{\nu}^1(X; B))$ and $A\in \mathcal{B}(\ell_{\nu}^{\infty}(X; B))$. Then also $A\in \mathcal{B}(\ell_{\nu}^p(X; B))$ for every $1 < p < \infty$.
\end{proposition}

The following simple lemma will be convenient later on.

\begin{lemma}\label{vectortoscalar}
Let $a_{kl} \in \mathbb{C}$ ($k,l\in X$) be given and $0<p\leq \infty$ be arbitrary. If the $\B(\Hil)$-valued matrix $A=[a_{kl} \cdot \mathcal{I}_{\B(\Hil)}]_{k,l\in X}$ defines an element in $\B(\ell^p_{\nu}(X;B))$ via matrix multiplication, then the scalar matrix $M=[a_{kl}]_{k,l\in X}$ defines an element in $\B(\ell_{\nu}^p(X))$ via matrix multiplication.
\end{lemma}

\begin{proof}
Let $c=(c_l)_{l\in X} \in \ell_{\nu}^p(X)$ be arbitrary. Then for every unit-norm $h\in B$, $(h_l)_{l\in X} := (c_l h)_{l\in X}$ is contained in $\ell^p_{\nu}(X;B)$ and $\Vert (h_l)_{l\in X} \Vert_{\ell^p_{\nu}(X;B)} = \Vert c \Vert_{\ell_{\nu}^p(X)}$. Thus
\begin{flalign}
\Vert Mc\Vert_{\ell_{\nu}^p(X)} &= \left\Vert \left(\left\vert \sum_{l\in X} a_{kl} c_l \right\vert \Vert h \Vert  \right)_{k\in X} \right\Vert_{\ell_{\nu}^p(X)} \notag \\
&= \left\Vert \left(\left\Vert \sum_{l\in X} a_{kl} c_l h \right\Vert \right)_{k\in X} \right\Vert_{\ell_{\nu}^p(X)} \notag \\
&= \Vert A(h_l)_{l\in X}\Vert_{\ell^p_{\nu}(X;B)} \notag \\
&\leq \Vert A \Vert_{\B(\ell^p_{\nu}(X;B))}\Vert c \Vert_{\ell_{\nu}^p(X)},\notag  
\end{flalign}
as was to be shown.
\end{proof}

We will also need the following special case of the Stein-Weiss interpolation theorem \cite[p.115, Theorem 5.4.1]{belö76}. 

\begin{proposition}[Stein-Weiss interpolation]\label{SteinWeiss}
Let $0<p\leq \infty$ be arbitrary. If $A \in \B(\ell^p_{\nu}(X))$ and $A \in \B(\ell^p_{1/\nu}(X))$ then also $A \in \B(\ell^p(X))$.      
\end{proposition}

Finally, we will often encounter $\B(\Hil)$-valued matrices $A=[A_{k,l}]_{k,l\in X}$, which define bounded operators on (weighted) Bochner sequence spaces with values in $\Hil$. If $A=[A_{k,l}]_{k,l\in X} \in \B(\ell^p_{\nu}(X;\Hil))$ (for some $1\leq p < \infty$), then its Banach space adjoint $A'$ is an element in $\B(\ell^q_{1/\nu}(X;\Hil))$ (where $\frac{1}{p}+\frac{1}{q}=1$) and a direct computation shows that 
\begin{equation}\label{BSadjoint}
A' = ([A_{k,l}^*]_{k,l\in X})^T
\end{equation}
where the exponent $T$ denotes matrix transposition and $A_{k,l}^*$ is the adjoint of $A_{k,l}\in \B(\Hil)$.

\subsection{Operator-valued frames}\label{g-frames}

A countable family $(T_k)_{k\in X}$ of bounded operators $T_k \in \B(\Hil)$ constitutes an \emph{operator-valued frame} (originally called \emph{generalized frame}, or simply \emph{g-frame}) \cite{Sun2006437} for $\Hil$, if there exist positive constants $0<A\leq B<\infty$, such that 
\begin{equation}\label{gframedef}
    A\Vert f \Vert^2 \leq \sum_{k\in X} \Vert T_k f \Vert^2 \leq B\Vert f \Vert^2 \qquad(\forall f\in \Hil).
\end{equation}
From now on, for brevity reason, we call an operator-valued frame a \emph{g-frame}. We will refer to the constants $A$ and $B$ as the \emph{lower} and \emph{upper g-frame bound}, respectively. 

We call $(T_k)_{k\in X}$ a \emph{g-Bessel sequence}, whenever the upper (but not necessarily the lower) inequality in (\ref{gframedef}) is satisfied for some prescribed $B>0$ and all $f\in \Hil$. Let $T=(T_k)_{k\in X}$ be a g-Bessel sequence in $\Hil$. Then the following operators are well-defined and bounded \cite{Sun2006437}:
\begin{itemize}
    \item The \emph{synthesis operator} $D_T : \ell^2(X;\Hil) \longrightarrow \mathcal{H}$, defined by $$D_T (f_k)_{k \in X} = \sum_{k \in X} T^*_k f_k ,$$
    \item The \emph{analysis operator} 
    $C_T : \mathcal{H} \longrightarrow \ell^2(X;\Hil)$, defined by $$C_T f = ( T_k f )_{k \in X} ,$$
    \item The \emph{g-frame operator}
    $S_T := D_T C_T: \mathcal{H} \longrightarrow \mathcal{H}$, given by $$S_T f = \sum_{k\in X} T^*_k T_k f .$$
    \item The \emph{g-Gram matrix} $G_T := C_T D_T : \ell^2(X;\Hil) \longrightarrow \ell^2(X;\Hil)$.
\end{itemize}
These operators share the same properties with their frame-theoretic pendants \cite{ole1n}. More precisely \cite{Sun2006437}, $T$ being a g-Bessel sequence with bound $B$ implies that $D_T$ and $C_T$ are adjoint to one another and $\Vert D_T \Vert_{\B(\ell^2(X;\Hil),\Hil)} = \Vert C_T \Vert_{\B(\Hil, \ell^2(X;\Hil))} \leq \sqrt{B}$. In particular, $\Vert T_k\Vert\leq \sqrt{B}$ for all $k\in X$. Furthermore, both $S_T$ and $G_T$ are self-adjoint and bounded by $B$ in this case. If $T$ is a g-frame, then $C_T$ is bounded, injective and has closed range, $D_T$ is bounded and surjective, and $S_T$ is bounded, self-adjoint, positive and invertible satisfying $A\cdot \mathcal{I}_{\Hil}\leq S_T \leq B\cdot \mathcal{I}_{\Hil}$ (in the sense of positive operators). Consequently, $S_T$ may be composed with its inverse (and vice versa) which yields the possibility of \emph{g-frame reconstruction} via  
\begin{equation}\label{gframerec}
f = \sum_{k\in X} T^*_k T_k S_T^{-1} f = \sum_{k\in X} S_T^{-1} T^*_k T_k f \qquad (\forall f\in \mathcal{H}) .
\end{equation}
The family $\widetilde{T} = (\widetilde{T}_k)_{k\in X}:=(T_k S_T^{-1})_{k\in X}$, which appears in the reconstruction process (\ref{gframerec}), is again a g-frame and called the \emph{canonical dual g-frame}. In particular, if $A\leq B$ are g-frame bounds for $T$, then $B^{-1} \leq A^{-1}$ are g-frame bounds for $\widetilde{T}$ and $B^{-1} \leq \Vert S_T^{-1} \Vert \leq A^{-1}$. More generally \cite{kutpatphi17}, if $(T^d_k)_{k\in X}$ is a g-Bessel sequence such that 
\begin{equation}\label{dualgframe}
f = \sum_{k\in X} T^*_k T^d_k f = \sum_{k\in X} (T^d_k)^* T_k f \qquad (\forall f\in \mathcal{H}), 
\end{equation}
or, equivalently, 
\begin{equation}\label{dualgframeoperatornotation}
\mathcal{I}_{\B(\Hil)} = D_T C_{T^d} = D_{T^d} C_T , 
\end{equation}
then $(T^d_k)_{k\in X}$ is already a g-frame and called a \emph{dual g-frame} (or simply \emph{dual}) of $(T_k)_{k\in X}$. 

\begin{remark}\label{matrixDC}
Assume that $T = (T_k)_{k\in X}$ is a g-Bessel sequence in $\Hil$. Then $C_T \in \B(\Hil, \ell^2(X;\Hil))$, $D_T \in \B(\ell^2(X;\Hil) , \Hil)$ and $G_T\in \B(\ell^2(X;\Hil))$ are all bounded. Since we can interpret $\Hil$ as a direct sum of Hilbert spaces indexed by just a singleton, we see that each of the operators $C_T$, $D_T$ and $G_T$ is a bounded operator from one direct sum of Hilbert spaces to another direct sum of Hilbert spaces. Hence, we are in the setting of the matrix calculus from \cite{Maddox:101881} (see also \cite[Section 3.1]{koebacasheihomosha23} for more details), which guarantees that every bounded operator $A$ between two direct sums of Hilbert spaces can be uniquely represented by a $\B(\Hil)$-valued matrix $\mathbb{M}(A) = [A_{k,l}]$, which acts on elements from a direct sum of Hilbert spaces, viewed as column vectors, via matrix-vector multiplication. Moreover, composition of such operators $A$ and $B$ corresponds to matrix multiplication, i.e. $\mathbb{M}(AB) = \mathbb{M}(A)\cdot \mathbb{M}(B)$, where each entry $[AB]_{k,l} = \sum_{n} A_{k,n} B_{n,l}$ of $\mathbb{M}(AB)$ converges in the operator norm topology with respect to $\B(\Hil)$ \cite{Maddox:101881,koebacasheihomosha23}. In particular, we have  
\begin{flalign}\label{gGrammatrix}
\mathbb{M}(D_T) &= \begin{bmatrix} \dots & T_{k-1}^* & T_k^* &  T_{k+1}^* & \dots \end{bmatrix} ,\notag \\
\mathbb{M}(C_T) &= \begin{bmatrix} \dots & T_{k-1} & T_k &  T_{k+1} & \dots \end{bmatrix}^T , \notag \\
\mathbb{M}(G_T) &= \begin{bmatrix} 
    \ddots & \vdots & \vdots & \, \\
    \dots & T_k T_l^* & T_k T_{l+1}^* & \dots \\
    \dots & T_{k+1} T_l^* & T_{k+1} T_{l+1}^* & \dots \\
    \, & \vdots & \vdots & \ddots \\ 
    \end{bmatrix} .
\end{flalign}
where the exponent $T$ denotes matrix transposition. Analogously, if $U = (U_k)_{k\in X}$ is another g-Bessel sequence in $\Hil$, then the canonical matrix representation of the \emph{mixed g-Gram matrix} $G_{T,U} := C_T D_U \in \B(\ell^2(X;\Hil))$ is the same as in (\ref{gGrammatrix}) after replacing each $T_l^*$ by $U_l^*$. 
\end{remark}

The above mentioned matrix calculus implies that we can identify $\B(\ell^2(X;\Hil))$ with a Banach *-algebra of $\B(\Hil)$-valued matrices, where the involution $^*$ is defined via $([A_{k,l}]_{k,l\in X})^* = [A^*_{l,k}]_{k,l\in X}$ and precisely corresponds to taking adjoints in $\B(\ell^2(X;\Hil))$, meaning that $\mathbb{M}(A^*) = \mathbb{M}(A)^*$ for all $A\in \B(\ell^2(X;\Hil))$. This viewpoint is particularly convenient for defining g-frame localization analogously as done in \cite{forngroech1} for frames, i.e. by g-Gram matrices belonging to some suitable sub-algebra $\A$ of $\B(\ell^2(X;\Hil))$ consisting of $\B(\Hil)$-valued matrices, see ahead.

\section{Localization of operator-valued frames}


\begin{definition}\label{spectralalgebra}
Let $\mathcal{A} = \A(X)$ be a unital Banach *-algebra of $\mathcal{B}(\Hil)$-valued matrices $A=[A_{k,l}]_{k,l\in X}$. Then $\mathcal{A}$ is called a \emph{solid spectral matrix algebra}, if
\begin{itemize}
    \item[(1.)] $\mathcal{A} \subseteq \mathcal{B}(\ell^2(X;\Hil))$, meaning that the inclusion map $\iota: \A \longrightarrow \mathcal{B}(\ell^2(X;\Hil))$ is a unital *-monomorphism (i.e. a faithful representation).
    \item[(2.)] $\mathcal{A}$ is \emph{inverse-closed} in $\mathcal{B}(\ell^2(X;\Hil))$, i.e., if $A \in \mathcal{A}$ defines an invertible operator in $\mathcal{B}(\ell^2(X;\Hil))$, then its inverse $A^{-1} \in \mathcal{B}(\ell^2(X;\Hil))$ is contained in $\mathcal{A}$ as well,
    \item[(3.)] $\mathcal{A}$ is \emph{solid}: if $A = [A_{k,l}]_{k,l \in X} \in \mathcal{A}$ and if $B = [B_{k,l}]_{k,l\in X}$ is a $\mathcal{B}(\Hil)$-valued matrix such that $\Vert B_{k,l} \Vert \leq \Vert A_{k,l} \Vert$ for all $k,l \in X$, then $B \in \mathcal{A}$ and $\Vert B \Vert_{\mathcal{A}} \leq \Vert A \Vert_{\mathcal{A}}$.
\end{itemize}
For brevity reason, we will call such an algebra $\A$ simply a \emph{spectral algebra}.
\end{definition}

Spectral algebras are also \emph{pseudo-inverse-closed} in $\B(\ell^2(X;\Hil))$:

\begin{theorem}\label{charlymain}\cite{forngroech1}
Let $\K$ be a Hilbert space. Suppose that $\A$ is a Banach *-algebra which is inverse-closed in $\B(\K)$ and let $M$ be a closed subspace of $\K$. If $A = A^* \in \A$, $\mathcal{N}(A) = M^{\perp}$ and $A:M\longrightarrow M$ is bounded and invertible, then the Moore-Penrose pseudo-inverse $A^{\dagger}$ of $A$ is an element in $\A$.
\end{theorem}

The above result holds also true for possibly non-self-adjoint operators $A$ \cite{forngroech1}. Since its proof was omitted by the authors, we present the details for completeness reason.

\begin{corollary}\label{pseudoinverseclosed}\cite{forngroech1}
Let $\A$ be a Banach *-algebra, which is inverse-closed in $\B(\K)$. Then $\A$ is also \emph{pseudo-inverse-closed} in $\B(\K)$, that is, if $A\in  \A$ defines a bounded operator in $\B(\K)$ with closed range, then $A^{\dagger} \in \A$.
\end{corollary}

\begin{proof}
By Lemma \ref{pseudoinverseformula}, the pseudo-inverse of $A$ is given by $A^{\dagger} = A^* (AA^*)^{\dagger}$. Since $AA^*$ is self-adjoint and $\mathcal{R}(AA^*) = \mathcal{R}(A)$ is closed, Theorem \ref{charlymain} implies that $(AA^*)^{\dagger} \in \A$. Consequently, $A^{\dagger} \in \A$ due to the algebra property of $\A$.
\end{proof}

\begin{ex}\label{spectralexamples}
Recently, the following examples of spectral algebras of $\B(\Hil)$-valued matrices indexed by a relatively separated set $X\subset \Rd$ haven been discussed \cite{koeba25}:

\noindent\emph{(1.)} For every $s>d$, the \emph{Jaffard algebra} $\J_s = \J_s(X)$ of $\B(\Hil)$-valued matrices $A = [A_{k,l}]_{k,l \in X}$ for which 
$$\sup_{k,l\in X} \Vert A_{k,l} \Vert (1+\vert k-l \vert)^s < \infty$$
is a spectral algebra.

\noindent\emph{(2.)} For every weight $\nu$ which
\begin{itemize}
    \item[(2a)] is of the form $\nu(x) = e^{\rho(\Vert x \Vert)}$, where $\rho:[0,\infty) \longrightarrow [0,\infty)$ is a continuous and concave function with $\rho(0) = 0$, and $\Vert \, . \, \Vert$ any norm on $\mathbb{R}^d$,
    \item[(2b)] satisfies the GRS-condition,
    \item[(2c)] satisfies the weak growth condition 
$$\nu(x) \geq C (1+\vert x \vert)^{\delta} \qquad \text{for some } \delta\in (0,1], C>0,$$
\end{itemize}
the weighted \emph{Schur algebra} $\MS^1_{\nu}=\MS^1_{\nu}(X)$ of $\B(\Hil)$-valued matrices $A = [A_{k,l}]_{k,l \in X}$ for which
$$\max \left\lbrace \sup_{k\in X} \sum_{l\in X} \Vert A_{k,l} \Vert \nu(k-l) \, , \, \sup_{l\in X} \sum_{k\in X} \Vert A_{k,l} \Vert \nu(k-l)\right\rbrace < \infty$$
is a spectral algebra.

\noindent\emph{(3.)} For every weight $\nu$, which is submultiplicative, symmetric, and satisfies the GRS-condition, the weighted \emph{Baskakov-Gohberg-Sjöstrand algebra} $\C_{\nu} = \C_{\nu}(\mathbb{Z}^d)$ of $\B(\Hil)$-valued matrices $A=[A_{k,l}]_{k,l\in \mathbb{Z}^d}$ for which
$$\Vert A \Vert_{\C_{\nu}}:= \sum_{l\in \mathbb{Z}^d} \sup_{k\in \mathbb{Z}^d} \Vert A_{k,k-l} \Vert \nu(l) <\infty$$
is a spectral algebra contained in $\B(\ell^2(\mathbb{Z}^d;\Hil))$.
 
\noindent\emph{(4.)} Spectral algebras $\J_{\nu}(X)$ which model more general off-diagonal decay than polynomial decay (and hence generalize $\J_s$).

\noindent\emph{(5.)} Anisotropic variations of all of the latter mentioned spectral algebras.
\end{ex}

With the above examples of spectral algebras in mind, we now give the main definition of this article, which extends the localization concept from \cite{forngroech1} to the g-frame setting. We remark that the special case $\A = \C_{\nu}(\mathbb{Z}^d)$ of Definition \ref{defloc} is \cite[Definition 5.6]{Krishtal2011}.

\begin{definition}\label{defloc}
Let $\A = \A(X)$ be a spectral algebra and $T=(T_k)_{k\in X}$ be a $\B(\Hil)$-valued family indexed by $X$. We say that $T$ is \emph{intrinsically $\A$-localized}, in short $T \sim_{\A} T$, if its associated $g$-Gram matrix $G_T = [T_k T_l^*]_{k,l\in X}$ belongs to $\A$. If $U = (U_k)_{k\in X}$ is another $\B(\Hil)$-valued family indexed by $X$, then we call $U$ and $T$ \emph{mutually $\A$-localized}, in short $U\sim_{\A} T$, if the mixed g-Gram matrix $G_{U,T} = [U_k T_l^*]_{k,l\in X}$ belongs to $\A$. 
\end{definition}

By the result below, we may always assume that intrinsically localized $\B(\Hil)$-valued sequences are g-Bessel sequences.

\begin{proposition}\label{locgBessel}
Every intrinsically localized $\B(\Hil)$-valued sequence is a g-Bessel sequence.
\end{proposition}

\begin{proof}
Assume that $T=(T_k)_{k\in X}$ is an intrinsically localized $\B(\Hil)$-valued family indexed by $X$. Then $G_T$ is contained in some spectral algebra $\A \subseteq \B(\ell^2(X;\Hil))$ and thus well-defined on $\ell^2(X;\Hil)$, bounded and self-adjoint. This implies that $D_T$ is bounded, since
$$\Vert D_T (h_l)_{l\in X} \Vert^2 = \langle G_T (h_l)_{l\in X} , (h_l)_{l\in X} \rangle_{\ell^2(X;\Hil)} \leq \Vert G_T\Vert_{\B(\ell^2(X;\Hil))} \Vert (h_l)_{l\in X} \Vert_{\ell^2(X;\Hil))}^2$$
holds true for all $(h_l)_{l\in X} \in \ell^{00}(X;\Hil)$ and, by density, also for all $(h_l)_{l\in X} \in \ell^2(X;\Hil)$. In particular, $C_T = D_T^*$ is bounded, which implies that $T$ is a g-Bessel sequence.
\end{proof}

\subsection{Intrinsic localization and duality}

Next, we will show that intrinsic localization of a g-frame is preserved by its canonical dual g-frame. Proving this relies heavily on the properties of a spectral algebra.

The following preparatory result is the g-frame analogue of \cite[Lemma 2.14]{hol22}.

\begin{lemma}\label{gGramfactorization}
Let $T=(T_k)_{k\in X}$ be a g-frame and $\widetilde{T} = (T_k S_T^{-1})_{k\in X}$ its canonical dual g-frame.
Then $G_T$ has closed range and
\begin{equation}\label{dualgGramformulas}
    G_{\widetilde{T}} = G_T^{\dagger} \qquad \text{and} \qquad G_{T,\widetilde{T}} = G_T G_T^{\dagger} .
\end{equation}
\end{lemma}

\begin{proof}
By definition of the analysis operator and the g-frame inequalities (\ref{gframedef}), $T=(T_k)_{k\in X}$ is a g-frame if and only if $C_T :\Hil \longrightarrow \ell^2(X;\Hil)$ is well-defined, bounded and bounded from below. This is equivalent to $C_T$ being a bounded, injective and closed-range operator and $D_T = C_T^*$ being bounded and surjective \cite{conw1}. In particular, $\R(G_T) = \R(C_T D_T) = \R(C_T)$ is closed and thus $G_T^{\dagger}$ is well-defined. 

In order to show the first relation in (\ref{dualgGramformulas}), we verify that $G_{\widetilde{T}}$ satisfies the three characterizing conditions of the pseudo-inverse of $G_T$ from Lemma \ref{pseudoinversechar}. Firstly, since $C_{\widetilde{T}} = C_T S_T^{-1}$, we see that $\N(G_{\widetilde{T}}) = \R(G_{\widetilde{T}})^{\perp} = \R(C_{\widetilde{T}})^{\perp} = \R(C_{T} S_T^{-1})^{\perp} = \R(C_T)^{\perp} = \R(G_T)^{\perp}$, where we also used the fact that g-Gram matrices associated to g-Bessel sequences are self-adjoint. Secondly, we similarly have $\R(G_{\widetilde{T}}) = \R(C_{\widetilde{T}}) = \R(C_{T} S_T^{-1}) = \R(C_T) = \R(G_T) = \N(G_T)^{\perp}$. Finally, we have $G_T G_{\widetilde{T}} G_T = C_T D_T C_{\widetilde{T}} D_{\widetilde{T}} C_T D_T = C_T D_T = G_T$ by (\ref{dualgframeoperatornotation}).

By using (\ref{dualgframeoperatornotation}) again, the second relation in (\ref{dualgGramformulas}) follows from the first via $G_T G_T^{\dagger} = G_T G_{\widetilde{T}} = C_T D_T C_{\widetilde{T}} D_{\widetilde{T}} = C_T D_{\widetilde{T}} = G_{T,\widetilde{T}}$.
\end{proof}

We are now prepared to prove our first main result, which is the g-frame theoretic analogue of \cite[Theorem 3.6]{forngroech1} and a generalization of \cite[Lemma 5.6]{Krishtal2011}, where the case $\A=\C_{\nu}(\mathbb{Z}^d)$ is treated. 

\begin{theorem}\label{duallocalized}
Let $T=(T_k)_{k\in X}$ be an intrinsically $\A$-localized g-frame. Then the following hold:
\begin{itemize}
    \item[(i)] The canonical dual g-frame $\widetilde{T} = (T_kS_T^{-1})_{k\in X}$ is intrinsically $\A$-localized.
    \item[(ii)] $T$ and $\widetilde{T}$ are mutually $\A$-localized.
\end{itemize}
\end{theorem}

\begin{proof}
By assumption we have $G_T \in \A$. By Lemma \ref{gGramfactorization}, $G_T^{\dagger}$ is well-defined, $G_{\widetilde{T}} = G_T^{\dagger}$ and $G_{T,\widetilde{T}} = G_T G_T^{\dagger}$. By Corollary \ref{pseudoinverseclosed}, $\A$ is pseudo-inverse-closed, hence $G_{\widetilde{T}} = G_T^{\dagger} \in \A$. Since $\A$ is an algebra, we obtain $G_{T,\widetilde{T}} = G_T G_T^{\dagger} \in \A$.
\end{proof}

From Theorem \ref{duallocalized}, we can derive the following analogue of \cite[Corollary 2.16]{hol22}, see also \cite[Lemma 2.13]{hol22}.

\begin{theorem}\label{equivalencerelation}
The relation $\sim_{\A}$ is symmetric. Furthermore, $\sim_{\A}$ is an equivalence relation on the set of all intrinsically $\A$-localized g-frames indexed by $X$.
\end{theorem}

\begin{proof}
If $G_{T,U} \in \A$, then $G_{U,T} = (G_{T,U})^* \in \A$, hence $\sim_{\A}$ is symmetric.

Clearly, $\sim_{\A}$ is a reflexive relation on the set of all intrinsically $\A$-localized g-frames on $\Hil$. Now, let $R = (R_k)_{k\in X}$, $T = (T_k)_{k\in X}$ and $U = (U_k)_{k\in X}$ be intrinsically $\A$-localized g-frames. To show transitivity, assume that $T\sim_{\A} U$ and $U\sim_{\A} R$. By Theorem \ref{duallocalized} we have  $T\sim_{\A}\widetilde{T}$ and $\widetilde{U} \sim_{\A} \widetilde{U}$, where $\widetilde{T}$ and $\widetilde{U}$ denote the canonical dual g-frames of $T$ and $U$ respectively. Thus, by g-frame reconstruction with respect to the dual frame pairs $(T,\widetilde{T})$ and $(U,\widetilde{U})$ we obtain via (\ref{dualgframeoperatornotation}) that
$$G_{T,R} = C_T D_R = C_T D_{\widetilde{T}} C_T D_U C_{\widetilde{U}} D_{\widetilde{U}} C_U D_R = G_{T,\widetilde{T}} G_{T,U} G_{\widetilde{U}} G_{U,R} \in \A ,$$
i.e. $T \sim_{\A} R$. 
\end{proof}

\subsection{Operator-valued frames and associated (quasi-) Banach spaces}

One of the most remarkable features of localized frames is that their associated frame reconstruction formulae not only hold true in the underlying Hilbert space, but also in a whole class of associated (quasi-)Banach spaces. In this section we show that analogous results hold true in the g-frame setting. In particular, we will consider the associated g-frame analysis and synthesis operators acting on other Banach spaces than $\Hil$ and $\ell^2(X;\Hil)$ respectively. In order to avoid tedious distinctions of these operators defined on different spaces, we will simply consider them as $\B(\Hil)$-valued matrices, as justified in Remark \ref{matrixDC}.

\subsubsection{Admissible weights}

One ingredient of the associated (quasi-)Banach spaces we will introduce are a certain class of weight functions, see below.

\begin{definition}\label{admissibleweightdef}\cite{forngroech1}
Let $\A$ be some subfamily of $\B(\ell^2(X;\Hil))$ (for instance, a spectral algebra) and $0<p_0\leq 1$. A weight function $\omega$ is called \emph{$(\A,p_0)$-admissible}, if every $A\in \A$ defines a bounded operator $\ell^p_{\omega}(X;\Hil) \longrightarrow \ell^p_{\omega}(X;\Hil)$ for every $p\in (p_0, \infty ] \cup \lbrace 1 \rbrace$. 
\end{definition}

\begin{ex}\label{admissibleex}
A weight being $(\A,p_0)$-admissible is a fairly strong assumption. However, recently the following examples of $(\A,p_0)$-admissible weights $\omega$ have been shown to exist \cite{koeba25}:

\noindent\emph{(1.)} If $s>d+r$ for some $r\geq 0$, and $\nu_r$ denotes the polynomial weight $\nu_r(x) = (1+\vert x \vert)^r$, then every $\nu_r$-moderate weight $\omega$ is $(\J_s,\frac{d}{s-r})$-admissible, where $\J_s$ denotes the Jaffard algebra from Example \ref{spectralexamples} (1.).

\noindent\emph{(2.)} If $\MS^1_{\nu}$ denotes the weighted \emph{Schur algebra} from Example \ref{spectralexamples} (2.), then every $\nu$-moderate weight $\omega$ is $(\MS^1_{\nu},1)$-admissible.

\noindent\emph{(3.)} If $\C_{\nu}$ denotes the Baskakov-Gohberg-Sjöstrand algebra from Example \ref{spectralexamples} (3.), then every $\nu$-moderate weight $\omega$ is $(\C_{\nu},1)$-admissible.
\end{ex}

The following result (whose scalar-analogue was stated in \cite{xxlgro14}) will be convenient later on.   

\begin{lemma}\label{1/wadmissible}
Let $\A$ be a spectral algebra and $\omega$ be $(\A,p_0)$-admissible. Then $1/\omega$ is $(\A,1)$-admissible.
\end{lemma}

\begin{proof}
Choose some arbitrary $A = [A_{k,l}]_{k,l\in X} \in \A$. Then $A^* = [A^*_{l,k}]_{k,l\in X} \in \A$ and in particular $A^* \in \B(\ell^p_{\omega}(X;\Hil))$ for $p\in [1,\infty)$ by assumption. Thus, by duality, the Banach space adjoint $(A^*)'$ of $A^*$, which, by (\ref{BSadjoint}), is given by $(A^*)' = (A^*)^* = A$, is contained in $\B(\ell^p_{1/\omega}(X;\Hil))$ for all $1<p\leq \infty$. 

It remains to show that $A\in \B(\ell^1_{1/\omega}(X;\Hil))$. Since $A^* = [[A^*]_{l,k}]_{k,l\in X} \in \A$, we have by solidity of $\A$ that the matrix $[\Vert [A^*]_{l,k}\Vert\cdot \mathcal{I}_{\B(\Hil)}]_{k,l\in X}$ is contained in $\A$ and thus $[\Vert [A^*]_{l,k}\Vert\cdot \mathcal{I}_{\B(\Hil)}]_{k,l\in X} \in \B(\ell_{\omega}^{\infty}(X;\Hil))$ by assumption. Lemma \ref{vectortoscalar} implies that the scalar matrix $M:=[\Vert [A^*]_{l,k} \Vert]_{k,l\in X}$ is contained in $\B(\ell_{\omega}^{\infty}(X))$. For $g=(g_l)_{l\in X} \in \ell^1_{1/\omega}(X;\Hil)$ this yields
\begin{flalign}
\Vert Ag \Vert_{\ell^1_{1/\omega}(X;\Hil)} 
&\leq \sum_{k\in X}\sum_{l\in X} \Vert A_{k,l} \Vert \Vert g_l \Vert \frac{1}{\omega(k)} \notag \\
&= \sum_{l\in X}\sum_{k\in X} \Vert A_{k,l} \Vert \Vert g_l \Vert \frac{1}{\omega(k)} \notag \\
&= \sum_{l\in X} \Vert g_l \Vert \frac{1}{\omega(l)} \sum_{k\in X} \Vert A_{k,l} \Vert  \frac{\omega(l)}{\omega(k)} \notag \\
&\leq \left( \sum_{l\in X} \Vert g_l \Vert \frac{1}{\omega(l)} \right) \sup_{m\in X}\sum_{k\in X} \Vert A_{k,m} \Vert \frac{\omega(m)}{\omega(k)}  \notag \\
&= \Vert g \Vert_{\ell^1_{1/\omega}(X;\Hil)} \sup_{m\in X}\left\vert \sum_{k\in X} \Vert [A^*]_{m,k} \Vert \frac{1}{\omega(k)} \right\vert \omega(m)  \notag \\
&= \Vert g \Vert_{\ell^1_{1/\omega}(X;\Hil)} \Vert M (w(k)^{-1})_{k\in X} \Vert_{\ell_{\omega}^{\infty}(X)} \notag \\
&\leq \Vert M \Vert_{\B(\ell_{\omega}^{\infty}(X))} \Vert g \Vert_{\ell^1_{1/\omega}(X;\Hil)} ,\notag
\end{flalign}
where we used that $\Vert (w(k)^{-1})_{k\in X} \Vert_{\ell_{\omega}^{\infty}(X)} = 1$ in the last step. Note that the interchange of order of summation in the second step is justified by the monotone convergence theorem, since only non-negative terms occur.
\end{proof}

\begin{corollary}\label{1admissible}
Let $\A$ be a spectral algebra. Whenever there exists some $(\A,p_0)$-admissible weight $\omega$ for some $0<p_0\leq 1$, the trivial weight $(1)_{l\in X}$ is $(\A,1)$-admissible.    
\end{corollary}

\begin{proof}
Let $A = [A_{k,l}]_{k,l\in X}\in \A$ be arbitrary. Then $[\Vert A_{k,l}\Vert\cdot \mathcal{I}_{\B(\Hil)}]_{k,l\in X}\in \A$ by solidity of $\A$. By assumption, Lemma \ref{1/wadmissible} and Lemma \ref{vectortoscalar} this implies that the scalar matrix $C:= [\Vert A_{k,l}\Vert]_{k,l\in X}$ defines an element in $\B(\ell^p_{\omega}(X))\cap \B(\ell^p_{1/\omega}(X))$ for each $1\leq p \leq \infty$. By Stein-Weiss interpolation (Proposition \ref{SteinWeiss}) this yields $C \in \B(\ell^p(X))$ for all $1\leq p\leq \infty$. Now, let $p\in [1,\infty)$ be arbitrary and take any $(f_l)_{l\in X} \in \ell^p(X;\Hil)$. Then $(\Vert f_l \Vert)_{l\in X}\in \ell^p(X)$ and thus
\begin{align}
    \sum_{k\in X} \left\Vert \sum_{l\in X} A_{k,l} f_l \right\Vert^p &\leq \sum_{k\in X} \left\vert \sum_{l\in X} \Vert A_{k,l}\Vert \Vert f_l \Vert \right\vert^p \notag \\
    &\leq \Vert C\Vert_{\B(\ell^p(X))}^p \Vert (f_l)_{l\in X} \Vert_{\ell^p(X;\Hil)}^p, \notag
\end{align}
which means that $A\in \B(\ell^p(X;\Hil))$. The modifications of the above for the case $p=\infty$ are obvious.
\end{proof}

\subsubsection{Definition of the spaces $\Hil^p_{\omega}(T^d,T)$ for $p<\infty$}

We are now able to define the spaces $\Hil^p_{\omega}(T^d,T)$. Let $\A$ be a spectral algebra and $\omega$ be an $(\A,p_0)$-admissible weight. Further, let $T=(T_k)_{k\in X}$ be a g-frame and $T^d = (T^d_k)_{k\in X}$ be any dual g-frame of $T$ such that $T^d$ and $T$ are mutually $\A$-localized (by Theorem \ref{duallocalized} we may choose the canonical dual g-frame of $T$ in case $T$ is intrinsically $\A$-localized). Consider the space 
$$\Hil^{00}(T) := D_T (\ell^{00}(X;\Hil)) = \left\lbrace \sum_{k\in K} T_k^* f_k:  K\subseteq X, \vert K \vert<\infty, f_k\in \Hil \, (\forall k\in K) \right\rbrace .$$
Since $D_T\in \B(\ell^2(X;\Hil), \Hil)$ and  $D_T(\ell^2(X;\Hil)) = \Hil$  for every g-frame $T$, $\Hil^{00}(T)$ is a dense subspace of $\Hil$, because $\ell^{00}(X;\Hil)$ is a dense subspace of $\ell^{2}(X;\Hil)$. For $p\in (p_0,\infty)\cup\lbrace 0,1 \rbrace$, we equip $\Hil^{00}(T)$ with the (quasi-)norm 
\begin{equation}\label{Hpwnorm}
    \Vert f \Vert_{\Hil^p_{\omega}(T^d,T)} := \Vert C_{T^d}f \Vert_{\ell_{\omega}^p(X;\Hil)} := \Vert (T^d_k f)_{k\in X} \Vert_{\ell_{\omega}^p(X;\Hil)}.
\end{equation}
Note that (\ref{Hpwnorm}) indeed defines a norm on $\Hil^{00}(T)$ for $1\leq p < \infty$ and $p=0$, and a quasi-norm on $\Hil^{00}(T)$ in case $p_0<p<1$. To see this, observe that for $f\in \Hil^{00}(T)$ we have $f=D_T(g_k)_{k\in X}$ for some $(g_k)_{k\in X} \in \ell^{00}(X;\Hil)$ and thus 
\begin{flalign}\label{CboundedonHpw}
\Vert C_{T^d}f \Vert_{\ell_{\omega}^p(X;\Hil)} &= \Vert  C_{T^d}D_T(g_k)_{k\in X} \Vert_{\ell_{\omega}^p(X;\Hil)} \notag \\
&= \Vert  G_{T^d,T}(g_k)_{k\in X} \Vert_{\ell_{\omega}^p(X;\Hil)} \notag \\
&\leq \Vert G_{T^d,T} \Vert_{\B(\ell_{\omega}^p(X;\Hil))}  \Vert (g_k)_{k\in X} \Vert_{\ell_{\omega}^p(X;\Hil)} <\infty,
\end{flalign}
since $G_{T^d,T} \in \A$ and $\omega$ is $(\A,p_0)$-admissible by assumption, whereas the other properties of a (quasi-)norm follow from $C_{T^d}$ being linear and injective as an operator $\Hil \longrightarrow \ell^2(X;\Hil)$ and $\Vert \, . \, \Vert_{\ell_{\omega}^p(X;\Hil)}$ being a (quasi-)norm.

This justifies the following definition.

\begin{definition}\label{Hpwdef}
Let $\A$ be a spectral algebra, $\omega$ be an $(\A,p_0)$-admissible weight, $T=(T_k)_{k\in X}$ be a g-frame and $T^d = (T^d_k)_{k\in X}$ be a dual g-frame of $T$ such that $T^d\sim_{\A} T$. Then for each $p\in (p_0,\infty)\cup\lbrace 1 \rbrace$ we define $\Hil^p_{\omega}(T^d,T)$ as the completion of $\Hil^{00}(T)$ with respect to the (quasi-)norm $\Vert \, . \, \Vert_{\Hil^p_{\omega}(T^d,T)}$. Furthermore we analogously define $\Hil^0_{\omega}(T^d,T)$ as the completion of $\Hil^{00}(T)$ with respect to the norm $\Vert \, . \, \Vert_{\Hil^{0}_{\omega}(T^d,T)} := \Vert C_{T^d}\, . \, \Vert_{\ell^{0}_{\omega}(X;\Hil)}=\Vert C_{T^d}\, . \, \Vert_{\ell^{\infty}_{\omega}(X;\Hil)}$.
\end{definition}
By definition, $\Hil^p_{\omega}(T^d,T)$ is a Banach space for each $p\in [1, \infty)\cup \lbrace 0 \rbrace$, and a quasi-Banach space for $p_0<p<1$. Moreover, $C_{T^d}$ extends via density to an isometry $\Hil^p_{\omega}(T^d,T) \longrightarrow \ell_{\omega}^p(X;\Hil)$ for $p\in (p_0,\infty)\cup\lbrace 0,1 \rbrace$. 

\subsubsection{Definition of the space $\Hil^{\infty}_{\omega}(T^d,T)$}

The definition of the space $\Hil^{\infty}_{\omega}(T^d,T)$, which we will consider next, is more technical. In order to motivate it, let $T^d = (T^d_k)_{k\in X}$ be a g-frame for $\Hil$ and consider the dense subspace $\Hil^{00}(T^d) := D_{T^d}(\ell^{00}(X;\Hil))$ of $\Hil$. Then it is easily seen that $(\Hil,\Hil^{00}(T^d))$ forms a dual pair with respect to the bilinear map $\Hil \times \Hil^{00}(T^d) \longrightarrow \mathbb{C}, (h,g)\mapsto \langle h,g \rangle$. 
Consequently, the family $\lbrace p_{g} : g \in \Hil^{00}(T^d) \rbrace$ of seminorms $p_g$ on $\Hil$, where $p_g(h) = \vert \langle g,h\rangle \vert$, defines a Hausdorff locally convex topology $\sigma(\Hil,\Hil^{00}(T^d))$ on $\Hil$. The topological completion $\overline{(\Hil, \sigma(\Hil,\Hil^{00}(T^d)))}$ consists of all equivalence classes $[\lbrace f_{\alpha} \rbrace]$ of Cauchy nets $\lbrace f_{\alpha} \rbrace \subset \Hil$, where (by definition) two Cauchy nets $\lbrace f_{\alpha}\rbrace$ and $\lbrace g_{\alpha} \rbrace$ are equivalent if and only if 
\begin{flalign}\label{Cauchyneteqr2}
\forall\varepsilon >0, \forall g \in \Hil^{00}(T^d): \exists \beta: \forall \alpha \geq \beta: \left\vert \left\langle g , f_{\alpha} - g_{\alpha} \right\rangle \right\vert < \varepsilon.
\end{flalign}
Since any $g\in \Hil^{00}(T^d)$ has the form $g = \sum_{k\in K} (T^d_k)^* h_k$ with $h_k\in \Hil$ for each $k\in X$ and $K\subset X$, $\vert K \vert < \infty$, (\ref{Cauchyneteqr2}) is equivalent to 
\begin{equation}\label{Cauchyneteqr}
\forall \varepsilon >0, \forall k\in X, \forall h\in \Hil: \exists \beta: \forall \alpha \geq \beta: 
\vert \langle h, T^d_k (f_{\alpha} - g_{\alpha}) \rangle \vert < \varepsilon.
\end{equation} 
Next, we consider the subspace of all equivalence classes of Cauchy \emph{sequences} $\lbrace f_n \rbrace_{n=1}^{\infty} \subset \Hil$ and denote this space by $\Hil^{\infty}(T^d,T)$. 

\noindent \textit{\textbf{Claim.} If $\lbrace f_n \rbrace_{n=1}^{\infty}$ and $\lbrace g_n \rbrace_{n=1}^{\infty}$ are equivalent Cauchy sequences, then the sequence $\lbrace f_n-g_n \rbrace_{n=1}^{\infty}$ is norm-bounded in $\Hil$.}

\begin{proof}
Recall, that $(\Hil^{00}(T^d), \Vert \cdot \Vert_{\Hil})$ is a dense subspace of $(\Hil, \Vert \cdot \Vert_{\Hil})$. By the Riesz representation theorem we may identify each $f_n \in \Hil$ with a linear functional $\ell_{f_n} \in \Hil^* \subseteq (\Hil^{00}(T^d))^*$. Hence, by (\ref{Cauchyneteqr2}), two such Cauchy sequences $\lbrace f_n \rbrace_{n=1}^{\infty}$ and $\lbrace g_n \rbrace_{n=1}^{\infty}$ being equivalent implies that the sequence of functionals $\lbrace \ell_{(f_n-g_n)} \rbrace_{n=1}^{\infty} \subset \Hil^* \subseteq (\Hil^{00}(T^d))^*$ is a weak* convergent sequence in $(\Hil^{00}(T^d))^*$ (with weak*-limit $0$). Therefore, by a general principle \cite{conw1}, $\lbrace \ell_{(f_n-g_n)} \rbrace_{n=1}^{\infty}$ is norm-bounded in $(\Hil^{00}(T^d))^*$ by some positive constant $c'$. We will show that $\lbrace \ell_{(f_n-g_n)} \rbrace_{n=1}^{\infty}$ is also norm-bounded in $\Hil^*$, which proves the claim. To this end, observe that $(\Hil^{00}(T^d), \Vert \cdot \Vert_{\Hil})$ being a dense subspace of $(\Hil, \Vert \cdot \Vert_{\Hil})$ implies via the reverse triangle inequality that the unit sphere of $\Hil^{00}(T^d)$ is a dense subset of the unit sphere of $\Hil$ (with respect to $\Vert \cdot \Vert_{\Hil}$). Thus, for every unit vector $h\in \Hil$ and every $n\in \mathbb{N}$, there exists a unit vector $h_n \in \Hil^{00}(T^d)$ such that $\Vert h - h_n\Vert < \Vert \ell_{(f_n-g_n)} \Vert_{\Hil^*}^{-1}$ (we may assume W.L.O.G. that each $\ell_{(f_n-g_n)}$ is non-zero). This yields for any unit-norm $h\in \Hil$ that 
$$\vert \ell_{(f_n-g_n)}(h) \vert \leq \vert \ell_{(f_n-g_n)}(h-h_n) \vert + \vert \ell_{(f_n-g_n)}(h_n) \vert \leq 1 + c' =: c$$
and the claim follows.
\end{proof}

So, now we look at the subspace $\Hil^{\infty}(T^d,T)$ of all equivalence classes of Cauchy \emph{sequences} $\lbrace f_n \rbrace_{n=1}^{\infty} \subset \Hil$ of the Hausdorff locally convex space $\overline{(\Hil, \sigma(\Hil,\Hil^{00}(T^d)))}$. Let $B$ denote the upper g-frame bound of $T^d$. By the proven claim and (\ref{Cauchyneteqr}), two such Cauchy sequences being equivalent implies 
$$\vert \langle h, T^d_k (f_n - g_n) \rangle \vert \leq \Vert h \Vert_{\Hil} \Vert T^d_k \Vert_{\B(\Hil)} \Vert f_n - g_n \Vert_{\Hil} \leq \sqrt{B} c \Vert h \Vert_{\Hil} \qquad (\forall k\in X, \forall n\in \mathbb{N}),$$
where $c$ denotes the norm-bound of the sequence $\lbrace \Vert f_n - g_n\Vert_{\Hil} \rbrace_{n=1}^{\infty}$. Thus
\begin{equation}\label{supconditionER}
    \sup_{n\in \mathbb{N}}\Vert C_{T^d}(f_n - g_n) \Vert_{\ell^{\infty}(X;\Hil)} < \infty .
\end{equation}
In the weighted case we 
consider a modified version of (\ref{supconditionER}) 
and arrive at the following (compare with \cite[Definition 3]{xxlgro14}). 

\begin{definition}{\label{ERHw}}
Let $\lbrace f_n \rbrace_{n=1}^{\infty}$ and $\lbrace g_n \rbrace_{n=1}^{\infty}$ be sequences in $\Hil$, $T^d=(T^d_k)_{k\in X}$ be a g-frame for $\Hil$ and $\omega$ be a weight. Then we say that $\lbrace f_n \rbrace_{n=1}^{\infty}$ and $\lbrace g_n \rbrace_{n=1}^{\infty}$ are \emph{$(T^d,\omega)$-equivalent}, in short $\lbrace f_n \rbrace_{n=1}^{\infty} \sim_{T^d,\omega} \lbrace g_n \rbrace_{n=1}^{\infty}$, if  
$$\forall k\in X: \lim_{n\rightarrow \infty} \Vert T^d_k(f_n-g_n)\Vert = 0 \qquad \text{and} \qquad \sup_{n\in \mathbb{N}}\Vert C_{T^d}(f_n - g_n) \Vert_{\ell_{\omega}^{\infty}(X;\Hil)} < \infty .$$
\end{definition}

Clearly, the relation $\sim_{T^d,\omega}$ defines an equivalence relation on the set of all sequences in $\Hil$. We define the space $\Hil_{\omega}^{\infty}(T^d,T)$ as a certain subspace of the vector space of equivalence classes corresponding to $\sim_{T^d,\omega}$ (compare with \cite[Definition 3]{xxlgro14}).

\begin{definition}\label{Hwinftydef}
Let $\A$ be a spectral algebra, $\omega$ be an $(\A,p_0)$-admissible weight, $T=(T_k)_{k\in X}$ be a g-frame and $T^d = (T^d_k)_{k\in X}$ be a dual g-frame of $T$ such that $T^d\sim_{\A} T$. Then $\Hil^{\infty}_{\omega}(T^d,T)$ is defined to be the space of all equivalence classes $f=[\lbrace f_n \rbrace_{n=1}^{\infty}]_{\sim_{T^d,\omega}}$ of sequences $\lbrace f_n \rbrace_{n=1}^{\infty}$ in $\Hil$ satisfying 
\begin{itemize}
    \item[(1.)] $\lim_{n\rightarrow \infty} T^d_k f_n =: T^d_k f$ exists in $\Hil$ for each $k\in X$, 
    \item[(2.)] $\sup_{n\in \mathbb{N}} \Vert C_{T^d}f_n \Vert_{\ell_{\omega}^{\infty}(X;\Hil)} < \infty .$
\end{itemize}
\end{definition}
It is quickly verified that the defining properties (1.) and (2.) do not depend on the choice of the representative sequence $\lbrace f_n \rbrace_{n=1}^{\infty}$. Condition (2.) can be rewritten as
\begin{equation}\label{Cineq}
\Vert T^d_k f_n \Vert \omega(k) \leq C \qquad (\forall k\in X, \forall n\in \mathbb{N})    
\end{equation}
for some positive constant $C$. Thus $\Vert T^d_k f \Vert \omega(k) = \lim_{n\rightarrow \infty} \Vert T^d_k f_n \Vert \omega(k) \leq C$ for each $k\in X$. In particular 
\begin{equation}\label{Hpinftynorm}
\Vert f \Vert_{\Hil^{\infty}_{\omega}(T^d,T)} := \Vert C_{T^d} f \Vert_{\ell_{\omega}^{\infty}(X;\Hil)} = \sup_{k\in X} \left\Vert \lim_{n\rightarrow \infty} T^d_k f_n \right\Vert \omega(k)
\end{equation}
is well-defined. Note that (\ref{Hpinftynorm}) certainly defines a semi-norm on $\Hil^{\infty}_{\omega}(T^d,T)$, since limits are linear and $\Vert \, . \, \Vert_{\ell_{\omega}^{\infty}(X;\Hil)}$ is a norm. If $\Vert f \Vert_{\Hil^{\infty}_{\omega}(T^d,T)} = 0$ for some $f=[\lbrace f_n \rbrace_{n=1}^{\infty}]_{\sim_{T^d, \omega}}$, then $\lbrace f_n \rbrace_{n=1}^{\infty} \sim_{T^d,\omega} \lbrace 0 \rbrace_{n=1}^{\infty}$ and thus $f=0$ in $\Hil^{\infty}_{\omega}(T^d,T)$. This shows that $\Vert \, . \,  \Vert_{\Hil^{\infty}_{\omega}(T^d,T)}$ actually defines a norm on $\Hil^{\infty}_{\omega}(T^d,T)$.

We explicitly emphasize that here $T^d_k f \in \Hil$ is the limit of the sequence $\lbrace T^d_k f_n \rbrace_{n=1}^{\infty} \subset \Hil$ (for each $k\in X$). In particular, $C_{T^d} f = (T^d_k f)_{k\in X}$ is understood as a pointwise limit. In this sense, $C_{T^d}:\Hil^{\infty}_{\omega}(T^d,T)\longrightarrow \ell_{\omega}^{\infty}(X;\Hil)$ is a well-defined isometry. 

\subsubsection{Properties of the spaces $\Hil_{\omega}^p(T^d,T)$}

In the previous subsection, we have shown that $C_{T^d}:\Hil^{\infty}_{\omega}(T^d,T)\longrightarrow \ell_{\omega}^{\infty}(X;\Hil)$ is a well-defined isometry. Below we show that its range is closed.  


\begin{lemma}\label{hard}
Under the same assumptions as in Definition \ref{Hwinftydef}, the map $f=[\lbrace f_n \rbrace_{n=1}^{\infty}]_{\sim_{T^d,\omega}} \mapsto (\lim_{n\rightarrow \infty} T^d_k f_n )_{k\in X} =: C_{T^d}f$ is an isometric isomorphism from $\Hil_{\omega}^{\infty}(T^d,T)$ onto the space $V = \lbrace g\in \ell_{\omega}^{\infty}(X;\Hil): g = G_{T^d,T} g \rbrace$. Furthermore, $V$ is a closed subspace of $\ell_{\omega}^{\infty}(X;\Hil)$, and thus $\Hil_{\omega}^{\infty}(T^d,T)$ a Banach space.
\end{lemma}

\begin{proof}
Let $f=[\lbrace f_n \rbrace_{n=1}^{\infty}]_{\sim_{T^d,\omega}} \in \Hil_{\omega}^{\infty}(T^d,T)$ be fixed. Then, by definition, we have that $\lim_{n\rightarrow \infty} T^d_k f_n =: T^d_k f$ exists in $\Hil$ for each $k\in X$ and that $C_{T^d}f = (T^d_k f)_{k\in X} \in \ell_{\omega}^{\infty}(X;\Hil)$. Since $\omega$ is $(\A,p_0)$-admissible and $T^d\sim_{\A}T$, we have $G_{T^d,T}\in \A \subset \B(\ell^{\infty}_{\omega}(X;\Hil))$, hence $G_{T^d,T} C_{T^d}f \in \ell^{\infty}_{\omega}(X;\Hil)$. We will show that $C_{T^d}f = G_{T^d,T} C_{T^d}f$ as elements in $\ell^{\infty}_{\omega}(X;\Hil)$. It suffices to verify the claimed identity component-wise, i.e. to show that 
$$\left\Vert T^d_k f - \sum_{l\in X} [G_{T^d,T}]_{k,l} T^d_l f \right\Vert = 0 \qquad (\text{for each }k\in X).$$
Since $T^d$ and $T$ is a pair of dual g-frames we have $G_{T^d,T} C_{T^d} f_n = C_{T^d} f_n$ for each $n\in \mathbb{N}$. Therefore, it suffices to show that $$\lim_{n\rightarrow \infty} \left\Vert \sum_{l\in X} [G_{T^d,T}]_{k,l} T^d_l f_n - \sum_{l\in X} [G_{T^d,T}]_{k,l} T^d_l f \right\Vert = 0 \qquad (\text{for each }k\in X).$$ 
The latter follows immediately from an application of the dominated convergence theorem for Bochner spaces \cite[Proposition 1.2.5]{HyNeVeWe16}, once we have verified that for each $k\in X$, there exists some family of non-negative numbers $(c^k_l)_{l\in X}$ such that
$$ \Vert [G_{T^d,T}]_{k,l} T^d_l f_n \Vert \leq  c_l^k  \quad (\forall n\in \mathbb{N}), \qquad \text{ and } \sum_{l\in X}  c_l^k  < \infty .$$
In order to show the latter dominance condition, we fix $k\in X$ and estimate via (\ref{Cineq})  
\begin{flalign}
\sum_{l\in X} \Vert [G_{T^d,T}]_{k,l} T^d_l f_n \Vert &\leq \sum_{l\in X} C \Vert [G_{T^d,T}]_{k,l}\Vert \cdot w(l)^{-1}\notag \\
&\leq Cw(k)^{-1} \sup_{m\in X}  \left\vert \sum_{l\in X} \Vert [G_{T^d,T}]_{m,l}\Vert \cdot w(l)^{-1} \right\vert w(m) =: (\ast). \notag     
\end{flalign}
Since $G_{T^d,T}= [T^d_k T_l^*]_{k,l\in X}$ is contained in $\A$, so is $[\Vert T^d_k T_l^* \Vert\cdot \mathcal{I}_{\B(\Hil)}]_{k,l\in X}$ due to solidity of $\A$. Thus, by Lemma \ref{vectortoscalar}, the scalar matrix $G:=[\Vert T^d_k T_l^* \Vert]_{k,l\in X}$ is contained in $\B(\ell_{\omega}^{\infty}(X))$ as well. Finally, since $\Vert (\omega(l)^{-1})_{l\in X}\Vert_{\ell_{\omega}^{\infty}(X)} = 1$, we arrive at
$$(\ast) = C \omega(k)^{-1} \Vert G\cdot(\omega(l)^{-1})_{l\in X} \Vert_{\ell_{\omega}^{\infty}(X)} \leq C\omega(k)^{-1} \Vert G \Vert_{\B(\ell_{\omega}^{\infty}(X))} < \infty$$
and the claim is proven. In particular, we have shown that $C_{T^d}$ is an isometry from $\Hil_{\omega}^{\infty}(T^d,T)$ into $V\subseteq \ell_{\omega}^{\infty}(X;\Hil)$.

To show that $C_{T^d}: \Hil_{\omega}^{\infty}(T^d,T) \longrightarrow V$ is onto, choose some arbitrary $g =( g_l)_{l\in X} \in V$. Let $F_1 \subseteq F_2\subseteq \dots$ be a nested sequence of finite subsets of $X$ such that $\bigcup_{n=1}^{\infty} F_n = X$ and set $f_n = \sum_{l\in F_n} T_l^* g_l \in \Hil$. Then
$$\lim_{n\rightarrow \infty} T^d_k f_n = \lim_{n\rightarrow \infty} \sum_{l\in F_n}T^d_k T_l^* g_l = \sum_{l\in X} [G_{T^d,T}]_{k,l} g_l = g_k \in \Hil$$
for each $k\in X$. In particular, $\sup_{k\in X} \Vert \lim_{n\rightarrow \infty} T^d_k f_n \Vert \omega(k) = \Vert g\Vert_{\ell_{\omega}^{\infty}(X;\Hil)} < \infty$. Consequently, $f:= [\lbrace f_n \rbrace_{n=1}^{\infty}]_{\sim_{T^d,\omega}} \in \Hil_{\omega}^{\infty}(T^d,T)$ and $C_{T^d}f = g$. 

Finally, let $\lbrace g_n \rbrace_{n=1}^{\infty}$ be a sequence in $V$, which converges in $\ell_{\omega}^{\infty}(X;\Hil)$ to some $g\in \ell_{\omega}^{\infty}(X;\Hil)$. Then
$$\lim_{n\rightarrow \infty} \Vert g - G_{T^d,T} g_n \Vert_{\ell_{\omega}^{\infty}(X;\Hil)} = \lim_{n\rightarrow \infty} \Vert g - g_n \Vert_{\ell_{\omega}^{\infty}(X;\Hil)} = 0$$
and at the same time 
$$\lim_{n\rightarrow \infty} \Vert G_{T^d,T} g - G_{T^d,T} g_n \Vert_{\ell_{\omega}^{\infty}(X;\Hil)} \leq \Vert G_{T^d,T} \Vert_{\B(\ell_{\omega}^{\infty}(X;\Hil))} \lim_{n\rightarrow \infty} \Vert g - g_n \Vert_{\ell_{\omega}^{\infty}(X;\Hil)} = 0.$$
Since limits in a Banach space are unique we must have $g = G_{T^d,T} g \in V$, hence $V$ is a closed subspace of $\ell_{\omega}^{\infty}(X;\Hil)$.
\end{proof}

So far we have shown the following. 

\begin{corollary}\label{isometrycor}
The space $\Hil_{\omega}^{p}(T^d,T)$ is a Banach space for each $p\in [1,\infty] \cup \lbrace 0 \rbrace$ and a quasi-Banach space for $p_0<p<1$. Moreover, $C_{T^d}$ is an isometry $\Hil_{\omega}^{p}(T^d,T) \longrightarrow \ell^p_{\omega}(X;\Hil)$ for all $p\in (p_0,\infty]\cup\lbrace 0,1 \rbrace$. 
\end{corollary}

\begin{remark}
Ignoring the weights for the moment, we explicitly emphasize the difference between the spaces $\Hil^{0}(T^d,T)$ and $\Hil^{\infty}(T^d,T)$. While $\Hil^{0}(T^d,T)$ is the completion of $\Hil^{00}(T)$ with respect to the norm $\Vert C_{T^d}\cdot \Vert_{\ell^{\infty}(X;\Hil)}$, the space $\Hil^{\infty}(T^d,T)$ is a complete normed subspace of the topological completion of $\Hil$ equipped with the $\sigma(\Hil,\Hil^{00}(T^d))$-topology. In particular, while $\Hil^{0}(T^d,T)$ consists of equivalence classes of sequences in $\Hil^{00}(T)$, with equivalence relation given by 
$$\lbrace f_n \rbrace_{n=1}^{\infty} \sim \lbrace g_n \rbrace_{n=1}^{\infty} \Leftrightarrow \lim_{n\rightarrow \infty} \sup_{k\in X} \Vert T^d_k (f_n-g_n)\Vert = 0,$$
the space $\Hil^{\infty}(T^d,T)$ consists of equivalence classes of sequences in $\Hil$, with equivalence relation given by 
\begin{flalign}
\lbrace f_n \rbrace_{n=1}^{\infty} \sim_{T^d} \lbrace g_n \rbrace_{n=1}^{\infty} \Leftrightarrow &\lim_{n\rightarrow \infty} \Vert T^d_k (f_n-g_n)\Vert = 0 \, \, (\forall k\in X), \notag \\
&\text{and } \,  \sup_{n\in \mathbb{N}} \sup_{k\in X} \Vert T^d_k (f_n - g_n)\Vert < \infty .\notag    
\end{flalign}
We will prove in Theorem \ref{dualspaces} that $\Hil^{\infty}(T^d,T)$ is isomorphic to the bi-dual $(\Hil^{0}(T^d,T))^{**}$ of the non-reflexive space $\Hil^{0}(T^d,T)$. In this sense $\Hil^{0}(T^d,T)$ can be identified with a closed linear subspace of $\Hil^{\infty}(T^d,T)$ \cite{conw1}.
\end{remark}

Next, we consider the properties of the synthesis operator in this setting. In order to keep the notation as simple as possible, we use the same symbol for each of the operators 
$$D_T:\ell_{\omega}^{p}(X;\Hil)\longrightarrow \Hil^{p}_{\omega}(T^d,T) \qquad (p\in (p_0,\infty)\cup \lbrace 0,1 \rbrace), $$
densely defined on $\ell^{00}(X;\Hil)$ via 
$$D_T(g_l)_{l\in X} = \sum_{l\in X} T_l^* g_l ,$$
and 
$$D_T:\ell_{\omega}^{\infty}(X;\Hil)\longrightarrow \Hil^{\infty}_{\omega}(T^d,T),$$
defined by the equivalence class 
$$D_T(g_l)_{l\in X} = \sum_{l\in X}T_l^* g_l := [\lbrace f_n \rbrace_{n=1}^{\infty}]_{\sim_{T^d,\omega}}$$ 
of the partial sums 
$$f_n = \sum_{l\in F_n}T_l^*g_l ,$$
where $F_1 \subseteq F_2\subseteq \dots$ is a nested sequence of finite subsets of $X$ such that $\bigcup_{n=1}^{\infty} F_n = X$. In fact, the latter definition is independent of the choice of $\lbrace F_n \rbrace_{n=1}^{\infty}$, as the next result shows. 


\begin{proposition}\label{Donto}
Let $\A$ be a spectral algebra, $\omega$ be an $(\A,p_0)$-admissible weight, $T=(T_k)_{k\in X}$ be a g-frame and $T^d = (T^d_k)_{k\in X}$ be a dual g-frame of $T$ such that $T^d\sim_{\A} T$. Then, for each $p\in (p_0,\infty]\cup \lbrace 0,1\rbrace$, the synthesis operator 
\begin{flalign}\label{syntheisseries}
D_T: \ell_{\omega}^{p}(X;\Hil) \longrightarrow  \Hil_{\omega}^{p}(T^d,T), \notag \\
(g_l)_{l\in X} \mapsto \sum_{l\in X} T_l^* g_l 
\end{flalign}
is bounded and surjective. For $p<\infty$ the series (\ref{syntheisseries}) converges unconditionally in $\Hil^p_{\omega}(T^d,T)$. In the case $p=\infty$, $D_Tg$ is independent of the choice of $\lbrace F_n \rbrace_{n=1}^{\infty}$ and the series (\ref{syntheisseries}) converges unconditionally in the $\sigma(\Hil,\Hil^{00}(T^d))$-topology. Furthermore, for all $p\in (p_0,\infty]\cup\lbrace 0,1 \rbrace$,
$$\Vert f \Vert_{\Hil_{\omega}^p} \asymp \inf\lbrace \Vert g \Vert_{\ell^p_{\omega}(X;\Hil)} : g\in \ell^p_{\omega}(X;\Hil), f = D_Tg \rbrace .$$
\end{proposition}

\begin{proof}
First, let $p \in (p_0,\infty)\cup\lbrace 0,1 \rbrace$. For $f=D_T(g_l)_{l\in X} \in \Hil^{00}(T)$ with $g=(g_l)_{l\in X}\in \ell^{00}(X;\Hil)$ we have $C_{T^d}f = G_{T^d,T} g$ and thus 
\begin{equation}\label{firstbounded}
\Vert f \Vert_{\Hil_{\omega}^p(T^d,T)} = \Vert G_{T^d,T} g \Vert_{\ell^p_{\omega}(X;\Hil)} \leq \Vert G_{T^d,T} \Vert_{\B(\ell^p_{\omega}(X;\Hil))}  \Vert g \Vert_{\ell^p_{\omega}(X;\Hil)}   
\end{equation}
by the $\A$-admissibility of $\omega$ and since $G_{T^d,T} \in \A$ by assumption. By density, the inequality (\ref{firstbounded}) extends to all $g\in \ell^p_{\omega}(X;\Hil)$, so $D_T$ is bounded. Now, if $\widetilde{f} = [\lbrace f_n \rbrace_{n=1}^{\infty}]\in \Hil_{\omega}^p(T^d,T)$ for some Cauchy sequence $\lbrace f_n \rbrace_{n=1}^{\infty}$ in $\Hil^{00}(T)$, then $\lbrace C_{T^d}f_n \rbrace_{n=1}^{\infty}$ is a Cauchy sequence in $\ell^p_{\omega}(X;\Hil)$, which has some limit $g\in \ell^p_{\omega}(X;\Hil)$. Writing $f:= D_Tg \in \Hil_{\omega}^p(T^d,T)$ we see that 
\begin{flalign}
\lim_{n\rightarrow \infty}\Vert f - f_n\Vert_{\Hil_{\omega}^p(T^d,T)} &= \lim_{n\rightarrow \infty} \Vert C_{T^d} (f - f_n) \Vert_{\ell^p_{\omega}(X;\Hil)} \notag \\
&\leq \Vert G_{T^d,T} \Vert_{\B(\ell^p_{\omega}(X;\Hil))} \lim_{n\rightarrow \infty} \Vert g - C_{T^d}f_n \Vert_{\ell^p_{\omega}(X;\Hil)} = 0, \notag
\end{flalign}
where we used that $G_{T^d,T}C_{T^d} = C_{T^d}$ holds true on $\Hil^{00}(T)$ due to $T^d$ and $T$ being a pair of dual g-frames. Consequently, $\widetilde{f} = f$ in $\Hil_{\omega}^p(T^d,T)$ and since $f = D_Tg$, we may conclude that $D_T:\ell^p_{\omega}(X;\Hil) \longrightarrow \Hil_{\omega}^p(T^d,T)$ is onto. The unconditional convergence of the series $D_Tg$ in $\Hil_{\omega}^p(T^d,T)$ is easily deduced from the latter argument.


Now consider the case $p=\infty$. Pick some arbitrary $g=(g_l)_{l\in X} \in \ell_{\omega}^{\infty}(X;\Hil)$ and assume that $F_1 \subseteq F_2\subseteq \dots$ is a nested sequence of finite subsets of $X$ such that $\bigcup_{n=1}^{\infty} F_n = X$. Then $f_n = \sum_{l\in F_n} T_l^* g_l$ is contained in $\Hil$ for each $n\in \mathbb{N}$ and 
$$\lim_{n\rightarrow \infty} T^d_k f_n = \lim_{n\rightarrow \infty} \sum_{l\in F_n}T^d_k T_l^* g_l = \sum_{l\in X} [G_{T^d,T}]_{k,l} g_l = [G_{T^d,T} g]_k \in \Hil$$
for each $k\in X$, since $G_{T^d,T} \in \A \subseteq \B(\ell_{\omega}^{\infty}(X;\Hil))$ by assumption. Furthermore, if $G$ denotes the scalar matrix $[\Vert T^d_k T_l^* \Vert]_{k,l\in X}$ as in the proof of Lemma \ref{hard}, we have that
\begin{flalign}\label{againG}
\sup_{n\in \mathbb{N}}\sup_{k\in X}\Vert T^d_k f_n \Vert \omega(k) &\leq \sup_{n\in \mathbb{N}}\sup_{k\in X} \sum_{l\in F_n} \Vert T^d_k T_l^* \Vert \Vert g_l \Vert \omega(k) \notag \\
&\leq \sup_{k\in X} \left\vert \sum_{l\in X} \Vert T^d_k T_l^* \Vert \Vert g_l \Vert \right\vert \omega(k) \notag \\
&\leq \Vert G \Vert_{\B(\ell^{\infty}_{\omega}(X))} \Vert g \Vert_{\ell_{\omega}^{\infty}(X;\Hil)}, 
\end{flalign}
where we used that $(\Vert g_l \Vert)_{l\in X}$ is a scalar sequence in $\ell_{\omega}^{\infty}(X)$ with norm equal to $\Vert g \Vert_{\ell_{\omega}^{\infty}(X;\Hil)}$.
Thus, $D_Tg = [\lbrace f_n \rbrace_{n=1}^{\infty}]_{\sim_{T^d,\omega}} \in \Hil_{\omega}^{\infty}(T^d,T)$ with $C_{T^d}(D_Tg) = G_{T^d,T}g$. The boundedness of $D_T$ now follows from 
$$\Vert D_Tg \Vert_{\Hil_{\omega}^{\infty}(T^d,T)} = \Vert G_{T^d,T} g \Vert_{\ell_{\omega}^{\infty}(X;\Hil)}  \leq \Vert G_{T^d,T} \Vert_{\B(\ell_{\omega}^{\infty}(X;\Hil))} \Vert g \Vert_{\ell_{\omega}^{\infty}(X;\Hil)}.$$
Next, let 
$G_1 \subseteq G_2\subseteq \dots$ be another nested sequence of finite subsets of $X$ with $\bigcup_{n=1}^{\infty} G_n = X$. We will show that $\lbrace \sum_{l\in G_n}T_l^* g_l \rbrace_{n=1}^{\infty} \sim_{T^d,\omega} \lbrace \sum_{l\in F_n}T_l^* g_l \rbrace_{n=1}^{\infty}$, which, by (\ref{Cauchyneteqr}), implies that both sequences belong to the same equivalence class in $\overline{(\Hil, \sigma(\Hil, \Hil^{00}(T^d)))}$, yielding the unconditional convergence of $D_Tg$ in the $\sigma(\Hil, \Hil^{00}(T^d))$-topology. Setting $H_n = F_n \cup G_n$ ($n\in \mathbb{N}$) gives a nested sequence of finite subsets whose union is $X$ as well. Then, 
for each $k\in X$,
\begin{flalign}
\lim_{n\rightarrow \infty} \left\Vert T^d_k \left(\sum_{l\in H_n} T_l^* g_l - \sum_{l\in F_n} T_l^* g_l \right) \right\Vert &= \lim_{n\rightarrow \infty} \left\Vert \sum_{l\in H_n\setminus F_n}  T^d_k T_l^* g_l \right\Vert \notag \\
&\leq \lim_{n\rightarrow \infty} \sum_{l\in X\setminus F_n} \Vert T^d_k T_l^* g_l \Vert \notag \\
&\leq \Vert g \Vert_{\ell^{\infty}_{\omega}(X;\Hil)} \lim_{n\rightarrow \infty} \sum_{l\in X\setminus F_n} \Vert T^d_k T_l^* \Vert \omega(l)^{-1} \notag \\
&=0. \notag
\end{flalign}
The last step follows from convergence of $\sum_{l\in X} \Vert T^d_k T_l^* \Vert \omega(l)^{-1}$ for each $k\in X$, which we already showed in ($\ast$) in the proof of Lemma \ref{hard}. Clearly, it also holds 
$$\sup_{n\in \mathbb{N}}\left\Vert C_{T^d} \left( \sum_{l\in H_n} T_l^* g_l - \sum_{l\in F_n} T_l^* g_l \right) \right\Vert_{\ell^{\infty}_{\omega}(X;\Hil)} < \infty.$$
Thus we have shown that $\lbrace \sum_{l\in H_n}T_l^* g_l \rbrace_{n=1}^{\infty} \sim_{T^d,\omega} \lbrace \sum_{l\in F_n}T_l^* g_l \rbrace_{n=1}^{\infty}$. 
The same argument holds if we replace each $F_n$ by $G_n$. Therefore, $\lbrace \sum_{l\in G_n}T_l^* g_l \rbrace_{n=1}^{\infty} \sim_{T^d,\omega} \lbrace \sum_{l\in F_n}T_l^* g_l \rbrace_{n=1}^{\infty}$, since $\sim_{T^d,\omega}$ is an equivalence relation. Finally, we show that $D_T$ is onto. Choose some arbitrary $f=[\lbrace f_n \rbrace_{n=1}^{\infty}]_{\sim_{T^d,\omega}} \in \Hil_{\omega}^{\infty}(T^d,T)$. Then $C_{T^d}f \in \ell_{\omega}^{\infty}(X;\Hil)$ by Corollary \ref{isometrycor} and $D_TC_{T^d}f \in \Hil_{\omega}^{\infty}(T^d,T)$ by the above. We claim that $f=D_TC_{T^d}f = [\lbrace \sum_{l\in F_n} T_l^* (T^d_l f) \rbrace_{n=1}^{\infty}]_{\sim_{T^d,\omega}}$ in $\Hil_{\omega}^{\infty}(T^d,T)$. To prove the claim, it suffices to show that $\lbrace \sum_{l\in F_n} T_l^* (T^d_l f) \rbrace_{n=1}^{\infty} \sim_{T^d,\omega} \lbrace f_n \rbrace_{n=1}^{\infty}$. Now, due to Lemma \ref{hard} we have   $C_{T^d}(D_TC_{T^d}f) = G_{T^d,T} C_{T^d}f = C_{T^d}f$, hence
$$\lim_{n\rightarrow \infty} \left\Vert T^d_k \left( \sum_{l\in F_n} T_l^* (T^d_l f) - f_n \right) \right\Vert = \left\Vert \sum_{l\in X} T^d_k T_l^* (T^d_l f) - T^d_k f \right\Vert = 0 \qquad (\forall k\in X).$$
By repeating the routine (\ref{againG}), we also see that  
$$\sup_{n\in \mathbb{N}}\left\Vert C_{T^d}\left(\sum_{l\in F_n} T_l^* (T^d_l f ) - f_n \right)\right\Vert_{\ell_{\omega}^{\infty}(X;\Hil)} <\infty$$
and the claim is proven.

Finally, showing the corresponding norm equivalences is a standard routine, see e.g. the proof of \cite[Proposition 2.4]{forngroech1}.
\end{proof}

\begin{theorem}\label{Banachgframeexpansion}
Let $\A$ be a spectral algebra, $\omega$ be an $(\A,p_0)$-admissible weight, $T=(T_k)_{k\in X}$ be a g-frame and $T^d = (T^d_k)_{k\in X}$ be a dual g-frame of $T$ such that $T^d\sim_{\A} T$. Then for each $p\in (p_0,\infty]\cup \lbrace 0,1 \rbrace$ it holds
\begin{equation}\label{Banachgframerec1}
    f = \sum_{k\in X} T_k^* T^d_k f \qquad (\forall f\in \Hil_{\omega}^p(T^d,T)),
\end{equation}
with unconditional convergence in $\Hil_{\omega}^p(T^d,T)$ for $p<\infty$ 
and unconditional convergence in the $\sigma(\Hil, \Hil^{00}(T^d))$-topology in the case $p=\infty$ respectively.
\end{theorem}

\begin{proof}
The case $p=\infty$ was already treated in the proof of Proposition \ref{Donto}. So, let $p\in (p_0,\infty)\cup \lbrace 0,1 \rbrace$. By Proposition \ref{Donto} and Corollary \ref{isometrycor} the composition $D_T C_{T^d}$ defines a bounded operator on $\Hil_{\omega}^p(T^d,T)$ whose associated series converges unconditionally. Moreover, (\ref{Banachgframerec1}) holds true on the subspace $\Hil^{00}(T)$ of $\Hil$ since $T^d$ and $T$ is a pair of dual g-frames. Hence, by density of $\Hil^{00}(T)$ in $\Hil_{\omega}^p(T^d,T)$ and boundedness of $D_T C_{T^d}$, this identity extends to all of $\Hil_{\omega}^p(T^d,T)$.
\end{proof}

So far, the spaces $\Hil_{\omega}^p(T^d,T)$ seem to be heavily dependent on the choice of the dual g-frame pair $(T^d,T)$. However, we will see that $\Hil_{\omega}^p(T^d,T)$ is independent of the choice of the dual g-frame pair $(T^d,T)$ as long as we stay in the same localization class $\sim_{\A}$. Before we give a precise formulation of the latter statement, we present an alternative description of the spaces $\Hil_{\omega}^p(T^d,T)$ for $p<\infty$. 

\begin{proposition}\label{Vpw}
Let $\A$ be a spectral algebra, $\omega$ be an $(\A,p_0)$-admissible weight, $T=(T_k)_{k\in X}$ be a g-frame and $T^d = (T^d_k)_{k\in X}$ be a dual g-frame of $T$ such that $T^d\sim_{\A} T$. Then for each $p \in (p_0,\infty)\cup\lbrace 0,1 \rbrace$, $\Hil_{\omega}^p(T^d,T)$ is the completion of the space 
$$V_{\omega}^p(T^d,T) = \lbrace f\in \Hil : C_{T^d}f \in \ell^{p}_{\omega}(X;\Hil) \rbrace$$
with respect to $\Vert \, . \, \Vert_{\Hil_{\omega}^p(T^d,T)}$.
\end{proposition}

\begin{proof}
Let $f\in V_{\omega}^p(T^d,T)$ be arbitrary. Then, by dual g-frame reconstruction, $f= D_T C_{T^d} f = \sum_{k\in X} T_k^* T^d_k f$ with unconditional convergence in $\Hil$. This series converges unconditionally in $\Hil_{\omega}^p(T^d,T)$ as well, since $C_{T^d}f \in \ell^{p}_{\omega}(X;\Hil)$ by assumption and thus $f=D_T C_{T^d} f \in \Hil_{\omega}^p(T^d,T)$ by Proposition \ref{Donto} and Theorem \ref{Banachgframeexpansion}. Hence $V_{\omega}^p(T^d,T) \subseteq \Hil_{\omega}^p(T^d,T)$. At the same time, we clearly have that $\Hil^{00}(T) = D_T(\ell^{00}(X;\Hil)) \subseteq V_{\omega}^p(T^d,T)$. Since $\Hil^{00}(T)$ is dense in $\Hil_{\omega}^p(T^d,T)$ by definition, the claim follows.
\end{proof}

For exponents $p<\infty$ and weights $\omega$ chosen in such a way that $\ell^p_{\omega}(X;\Hil)$ is continuously embedded in $\ell^2(X;\Hil)$, we obtain the following consequence of the latter, which is the g-frame theoretic version of \cite[Proposition 2.3]{forngroech1}. 

\begin{corollary}\label{easydescription}
Let $\A$ be a spectral algebra, $\omega$ be an $\A$-admissible weight, $T=(T_k)_{k\in X}$ be a g-frame and $T^d = (T^d_k)_{k\in X}$ be a dual g-frame of $T$ such that $T^d\sim_{\A} T$. Let $p<\infty$ and $\omega$ be so that $\ell^p_{\omega}(X;\Hil)$ is continuously embedded in $\ell^2(X;\Hil)$. Then
$$\Hil_{\omega}^p(T^d,T) = V_{\omega}^p(T^d,T) .$$
\end{corollary}

\begin{proof}
By Proposition \ref{Vpw}, it remains to show that $\Hil_{\omega}^p(T^d,T) \subseteq V_{\omega}^p(T^d,T)$. Let $\widetilde{f}=[\lbrace f_n \rbrace_{n=1}^{\infty}] \in \Hil_{\omega}^{p}(T^d,T)$ where $\lbrace f_n \rbrace_{n=1}^{\infty} \subseteq V_{\omega}^p(T^d,T)$ is a representative Cauchy sequence in the $\Vert \, . \, \Vert_{\Hil_{\omega}^{p}(T^d,T)}$-norm. If $A$ denotes the lower g-frame bound of $T^d$ and $c$ the bound coming from the continuous embedding of $\ell^p_{\omega}(X;\Hil)$ in $\ell^2(X;\Hil)$, then  
\begin{flalign}
\Vert f_n - f_m \Vert &\leq A^{-1/2} \Vert C_{T^d}(f_n - f_m) \Vert_{\ell^2(X;\Hil)} \notag \\
&\leq A^{-1/2}c \Vert C_{T^d}(f_n - f_m) \Vert_{\ell^p_{\omega}(X;\Hil)} = A^{-1/2}c\Vert f_n - f_m \Vert_{\Hil_{\omega}^{p}(T^d,T)}. \notag
\end{flalign}
Consequently $\lbrace f_n \rbrace_{n=1}^{\infty}$ is a Cauchy sequence in $\Hil$ and thus convergent in $\Hil$ to some $f\in \Hil$. Note that this limit $f$ is independent of our choice of the representative sequence of $\widetilde{f}$. Then, $\lbrace C_{T^d}f_n \rbrace_{n=1}^{\infty}$ is a Cauchy sequence in both $\ell^p_{\omega}(X;\Hil)$ and $\ell^2(X;\Hil)$. By passing onto their respective $\Hil$-valued components, we see that these Cauchy sequences must converge to the same limiting $\Hil$-valued sequence $(h_k)_{k\in X}$. However, by continuity of $C_{T^d}:\Hil\longrightarrow \ell^2(X;\Hil)$, we must have $(h_k)_{k\in X} = C_{T^d} f$. Thus $C_{T^d} f \in \ell^p_{\omega}(X;\Hil) \subseteq \ell^2(X;\Hil)$ and $\widetilde{f} = D_T C_{T^d} f = f \in \Hil$, as was to be shown.
\end{proof}

\begin{remark}
The spaces $V^p(T,T^d)$ have also been studied in \cite{liliu22}, where the case $1\leq p \leq 2$ has been mainly considered. Moreover, in \cite[Example 3.1]{liliu22}, it has been shown that $V^p(T,T^d)$ ($1\leq p \leq 2$) might be the trivial space $\lbrace 0_{\Hil} \rbrace$ if the given g-frame $T$ is not suitably localized. 
\end{remark}

\begin{theorem}\label{equalityofspaces}
Let $\A$ be a spectral algebra and $\omega$ be an $(\A,p_0)$-admissible weight. Let $T=(T_k)_{k\in X}$ and $U=(U_k)_{k\in X}$ be g-frames and assume that $T^d = (T^d_k)_{k\in X}$ is a dual g-frame of $T$ such that $T^d\sim_{\A} T$, and that $U^d = (U^d_k)_{k\in X}$ is a dual g-frame of $U$ such that $U^d\sim_{\A} U$. If $T^d\sim_{\A} U$ and $U^d\sim_{\A} T$, then, for each $p\in (p_0,\infty]\cup \lbrace 0,1 \rbrace$,
$$\Hil_{\omega}^p(T^d,T) = \Hil_{\omega}^p(U^d,U)$$
with equivalent norms. 
\end{theorem}

\begin{proof}
Let $p\in (p_0,\infty)\cup\lbrace 0,1 \rbrace$. By Proposition \ref{Vpw}, $\Hil_{\omega}^p(T^d,T)$ and $\Hil_{\omega}^p(U^d,U)$ are the completions of $(V_{\omega}^p(T^d,T), \Vert \, . \, \Vert_{\Hil_{\omega}^p(T^d,T)})$ and $(V_{\omega}^p(U^d,U), \Vert \, . \, \Vert_{\Hil_{\omega}^p(U^d,U)})$ respectively. Now if $f\in V_{\omega}^p(T^d,T)$, then $C_{U^d} f =G_{U^d,T} C_{T^d} f$ by g-frame reconstruction in $\Hil$ and thus
$$\Vert C_{U^d} f \Vert_{\ell^p_{\omega}(X;\Hil)} \leq \Vert G_{U^d,T} \Vert_{\B(\ell^p_{\omega}(X;\Hil))} \Vert C_{T^d} f \Vert_{\ell^p_{\omega}(X;\Hil)},$$
where we also used the assumption  $U^d\sim_{\A} T$ and the $\A$-admissibility of $\omega$. This implies that $V_{\omega}^p(T^d,T)$ is continuously embedded in $V_{\omega}^p(U^d,U)$. By switching the roles of $(T^d,T)$ and $(U^d,U)$, we may conclude that $V_{\omega}^p(U^d,U) = V_{\omega}^p(T^d,T)$ with equivalent norms. Consequently $\Hil_{\omega}^p(T^d,T) = \Hil_{\omega}^p(U^d,U)$ with equivalent norms. 

It remains to consider the case $p=\infty$. 

\noindent \emph{Step 1.} Recall from Lemma \ref{hard} that the spaces $V(T^d,T) = \lbrace g\in \ell_{\omega}^{\infty}(X;\Hil): g = G_{T^d,T}g \rbrace$ and $V(U^d,U) = \lbrace g\in \ell_{\omega}^{\infty}(X;\Hil): g = G_{U^d,U}g \rbrace$ are isometrically isomorphic to $\Hil_{\omega}^p(T^d,T)$ and $\Hil_{\omega}^p(U^d,U)$ via $C_{T^d}$ and $C_{U^d}$ respectively. Next, observe that the identities $G_{T^d,T} G_{T^d,U} = G_{T^d,U} = G_{T^d,U} G_{U^d,U}$ are valid in $\B(\ell^2(X;\Hil))$ due to g-frame reconstruction in $\Hil$ of the respective dual g-frame pairs $(T^d,T)$ and $(U^d,U)$. Since each of these occurring g-Gram matrices is contained in $\A$ by assumption, these identities are also valid in $\A$ and consequently in $\B(\ell_{\omega}^{\infty}(X;\Hil))$ by the $(\A,p_0)$-admissibility of $\omega$. This implies that both $G_{T^d,U}: V(U^d,U) \longrightarrow V(T^d,T)$ and $G_{U^d,T}: V(T^d,T) \longrightarrow V(U^d,U)$ are bounded. 

\noindent \emph{Step 2.} 
Next, note that  $C_{U^d} = G_{U^d,T} C_{T^d}$ and $C_{T^d} = G_{T^d,U} C_{U^d}$ as operators $\Hil\longrightarrow \ell^2(X;\Hil)$ by g-frame reconstruction of the dual g-frame pairs $(T^d,T)$ and $(U^d,U)$. By repeating the proof details of Lemma \ref{hard} with $G_{U^d,T}$ instead of $G_{T^d,T}$, we also obtain that $G_{U^d,T}C_{T^d} = C_{U^d}$ as an operator $\Hil_{\omega}^{\infty}(U^d,U)\longrightarrow V(U^d,U)$, and similarly, $G_{T^d,U}C_{U^d} = C_{T^d}$ as an operator $\Hil_{\omega}^{\infty}(T^d,T)\longrightarrow V(T^d,T)$.

\noindent \emph{Step 3.} We show that $[\lbrace f_n \rbrace_{n=1}^{\infty}]_{\sim_{T^d,\omega}} = [\lbrace f_n \rbrace_{n=1}^{\infty}]_{\sim_{U^d,\omega}}$ as sets for all sequences $\lbrace f_n \rbrace_{n=1}^{\infty}$ in $\Hil$. Assume that $\lbrace g_n \rbrace_{n=1}^{\infty}$ and $\lbrace h_n \rbrace_{n=1}^{\infty}$ are two sequences in $\Hil$ such that $\lbrace h_n \rbrace_{n=1}^{\infty} \sim_{T^d,\omega} \lbrace g_n \rbrace_{n=1}^{\infty}$. Then, by Step 2,
$$\sup_{n\in \mathbb{N}} \Vert C_{U^d} (g_n -h_n)\Vert_{\ell^{\infty}_{\omega}(X;\Hil)} \leq \Vert G_{U^d,T} \Vert_{\B(\ell^{\infty}_{\omega}(X;\Hil))} \sup_{n\in \mathbb{N}} \Vert C_{T^d} (g_n -h_n)\Vert_{\ell^{\infty}_{\omega}(X;\Hil)} < \infty .$$
Furthermore, since $\lbrace g_n - h_n \rbrace_{n=1}^{\infty} \sim_{T^d,\omega} \lbrace 0 \rbrace_{n=1}^{\infty}$, we have $[\lbrace g_n - h_n \rbrace_{n=1}^{\infty}]_{\sim_{T^d,\omega}} = [\lbrace 0 \rbrace_{n=1}^{\infty}]_{\sim_{T^d,\omega}} = 0\in \Hil^{\infty}_{\omega}(T^d,T)$, which, again by Step 2, yields
\begin{flalign}
\lim_{n\rightarrow \infty} \Vert U^d_k (g_n -h_n)\Vert \omega(k) &\leq \sup_{k\in X} \lim_{n\rightarrow \infty} \Vert U^d_k (g_n -h_n)\Vert \omega(k) \notag \\
&= \sup_{k\in X} \lim_{n\rightarrow \infty} \left\Vert \sum_{l\in X} U^d_k T_l^* T^d_l (g_n -h_n)\right\Vert \omega(k) \notag \\
&= \sup_{k\in X}  \left\Vert \sum_{l\in X} U^d_k T_l^* \lim_{n\rightarrow \infty} T^d_l (g_n -h_n)\right\Vert \omega(k) \notag \\
&= \Vert G_{U^d,T}C_{T^d} 0\Vert_{\ell^{\infty}_{\omega}(X;\Hil)} = 0 \qquad (\text{for each }k\in X), \notag
\end{flalign}
where the interchange of the limit and the sum is justified by the dominated convergence theorem as in the proof of Lemma \ref{hard}.
Thus 
$$\lim_{n\rightarrow \infty} \Vert U^d_k (g_n -h_n)\Vert =0 \qquad  (\forall k\in X),$$ and consequently, $\lbrace g_n \rbrace_{n=1}^{\infty} \sim_{U^d,\omega} \lbrace h_n \rbrace_{n=1}^{\infty}$. We conclude that $[\lbrace f_n \rbrace_{n=1}^{\infty}]_{\sim_{T^d,\omega}} \subseteq  [\lbrace f_n \rbrace_{n=1}^{\infty}]_{\sim_{U^d,\omega}}$ for all sequences $\lbrace f_n \rbrace_{n=1}^{\infty}$ in $\Hil$. The converse inclusion relation is shown analogously by switching the roles of $(T^d,T)$ and $(U^d,U)$. 

\noindent \emph{Step 4.} Let $f =[\lbrace f_n \rbrace_{n=1}^{\infty}]_{\sim_{T^d,\omega}}$ be arbitrary. Then $C_{T^d}f \in V(T^d,T)$ by Lemma \ref{hard}. By Step 3, we have that $[\lbrace f_n \rbrace_{n=1}^{\infty}]_{\sim_{T^d,\omega}}= [\lbrace f_n \rbrace_{n=1}^{\infty}]_{\sim_{U^d,\omega}} \in \Hil_{\omega}^{\infty}(T^d,T)$, and by Steps 1 and 2, we have that $C_{U^d} f = (\lim_{n\rightarrow \infty} U^d_k f_n)_{k\in X} = G_{U^d,T} C_{T^d}f \in V(U^d,U)$. Thus  
$$\Vert f \Vert_{\Hil_{\omega}^{\infty}(U^d,U)} = \Vert C_{U^d} f \Vert_{\ell_{\omega}^{\infty}(X;\Hil)} \leq \Vert G_{U^d,T} \Vert_{\B(\ell_{\omega}^{\infty}(X;\Hil))} \Vert C_{T^d} f \Vert_{\Hil_{\omega}^{\infty}(T^d,T)}.$$
This shows that $\Hil_{\omega}^{\infty}(T^d,T)$ is continuously embedded in $\Hil_{\omega}^{\infty}(U^d,U)$. Analogously, we obtain that $\Hil_{\omega}^{\infty}(U^d,U)$ is continuously embedded in $\Hil_{\omega}^{\infty}(T^d,T)$ after switching the roles of $(T^d,T)$ and $(U^d,U)$. This completes the proof.
%
%
\end{proof}

Via the above theorem, we can extend Theorem \ref{Banachgframeexpansion} as follows if $T$ and $T^d$ also are intrinsically localized. 

\begin{theorem}\label{reconstructioncorollary}
Let $\A$ be a spectral algebra and $\omega$ be an $(\A,p_0)$-admissible weight. Let $T=(T_k)_{k\in X}$ be an intrinsically $\A$-localized g-frame and assume that $T^d = (T^d_k)_{k\in X}$ is an intrinsically $\A$-localized dual g-frame of $T$ such that $T^d\sim_{\A} T$. Then, for each $p\in (p_0,\infty]\cup \lbrace 0,1 \rbrace$,
$$\Hil_{\omega}^p(T^d,T) = \Hil_{\omega}^p(T,T^d)$$
with equivalent norms and 
$$f = \sum_{k\in X} T_k^* T^d_k f = \sum_{k\in X} (T^d_k)^* T_k f \qquad (\forall f \in \Hil_{\omega}^p(T^d,T))$$
with unconditional convergence in $\Hil_{\omega}^p(T^d,T)$ for $p<\infty$ and unconditional convergence in the $\sigma(\Hil, \Hil^{00}(T^d))$-topology in the case $p=\infty$ respectively.
\end{theorem}

\begin{proof}
The special case $(U^d,U)=(T,T^d)$ of Theorem \ref{equalityofspaces} combined with Theorem \ref{Banachgframeexpansion} yields the claim.
\end{proof}

Recall from Theorem \ref{duallocalized} that if $T$ is an intrinsically $\A$-localized g-frame, then $\widetilde{T} \sim_{\A} T$ and $\widetilde{T} \sim_{\A} \widetilde{T}$, where $\widetilde{T}$ denotes the canonical dual g-frame of $T$. Therefore, given any $(\A,p_0)$-admissible weight, the spaces $\Hil^p_{\omega}(\widetilde{T},T)$ ($p\in (p_0,\infty ]\cup \lbrace 0,1\rbrace$) are intrinsically defined and only depend on $T$. In particular, we may deduce the following result in this case. 

\begin{corollary}\label{canonicalduallocalized}
Let $\A$ be a spectral algebra and $\omega$ be an $(\A,p_0)$-admissible weight. If $T=(T_k)_{k\in X}$ is an intrinsically $\A$-localized g-frame, then for each $p\in (p_0,\infty]\cup \lbrace 0,1 \rbrace$,
$$\Hil_{\omega}^p(\widetilde{T},T) = \Hil_{\omega}^p(T,\widetilde{T})$$
with equivalent norms and 
\begin{equation}\label{dualgframeexploc}
    f = \sum_{k\in X} T^*_k T_k S_T^{-1} f = \sum_{k\in X} S_T^{-1} T^*_k T_k f \qquad (\forall f\in \Hil_{\omega}^p(\widetilde{T},T)),
\end{equation}
with unconditional convergence in $\Hil_{\omega}^p(\widetilde{T},T)$ for $p<\infty$ and unconditional convergence in the $\sigma(\Hil, \Hil^{00}(\widetilde{T}))$-topology in the case $p=\infty$ respectively.
\end{corollary}


\begin{remark}\label{Banachgframe}
In the language of \cite{Kho11}, Corollary \ref{canonicalduallocalized} says that every intrinsically localized g-frame is a Banach g-frame for $\Hil_{\omega}^p(\widetilde{T},T)$ for each $p\in [1,\infty] \cup \lbrace 0 \rbrace$.   
\end{remark}

We conclude this section with the following result on duality relations. Note that the spaces $\Hil^q_{1/\omega}(T^d,T)$ that occur below are well-defined, since, by Lemma \ref{1/wadmissible}, the weight $1/\omega$ is $(\A,1)$-admissible whenever $\omega$ is $(\A,1)$-admissible. 

\begin{theorem}\label{dualspaces}
Let $\A$ be a spectral algebra, $\omega$ be an $(\A,1)$-admissible weight, $T=(T_k)_{k\in X}$ be an intrinsically $\A$-localized g-frame and $T^d = (T^d_k)_{k\in X}$ be a dual g-frame of $T$ such that $T^d\sim_{\A} T$. Then
$$\big(\Hil^p_{\omega}(T^d,T)\big)^{*} \simeq \Hil^q_{1/\omega}(T,T^d),$$
for all $p\in [1,\infty)$ with $\frac{1}{p}+\frac{1}{q}=1$ and for $(p,q)=(0,1)$. In particular, 
$$\Hil^{\infty}_{\omega}(T^d,T) \simeq   \big(\Hil^0_{\omega}(T^d,T)\big)^{**}.$$
\end{theorem}

\begin{proof}
We separate three cases for $p$.

\noindent \emph{Case (a): $1<p<\infty$.} 

Observe that 
$$\beta:\Hil^p_{\omega}(T^d,T) \times \Hil^q_{1/\omega}(T,T^d) \longrightarrow \mathbb{C}, \qquad \beta(f,g) = \sum_{k\in X} \langle T^d_k f, T_k g \rangle$$ 
defines a bounded sesquilinear form. Indeed, by Hölder's inequality,
\begin{flalign}
\left\vert \beta(f,g) \right\vert &\leq \sum_{k\in X} \Vert T^d_k f \Vert \Vert T_k g \Vert \notag \\
&\leq \Vert C_{T^d} f \Vert_{\ell^p_{\omega}(X;\Hil)} \Vert C_{T} g \Vert_{\ell^q_{1/\omega}(X;\Hil)} \notag \\
&= \Vert f \Vert_{\Hil^p_{\omega}(T^d,T)} \Vert g \Vert_{\Hil^q_{1/\omega}(T,T^d)}. \notag
\end{flalign}
As a consequence, the map 
$$\kappa : \Hil^q_{1/\omega}(T,T^d) \longrightarrow (\Hil^p_{\omega}(T^d,T))^{*}, \quad g\mapsto \kappa(g),$$
where the action of $\kappa(g)$ on any $f\in \Hil^p_{\omega}(T^d,T)$ is given by $\kappa(g)(f) = \beta(f,g)$, defines a linear and bounded operator. We will prove that $\kappa$ is bijective.

In order to prove that $\kappa$ is surjective, let $\rho \in (\Hil^p_{\omega}(T^d,T))^{*}$ be arbitrary. Then $\rho \circ D_T \in (\ell^p_{\omega}(X;\Hil))^{*}$ due to Proposition \ref{Donto}. Since $(\ell^p_{\omega}(X;\Hil))^{*} \cong \ell^q_{1/\omega}(X;\Hil)$, there exists a unique $(g_k)_{k\in X} \in \ell^q_{1/\omega}(X;\Hil)$ with $\rho(D_T (h_k)_{k\in X}) = \sum_{k\in X} \langle h_k, g_k \rangle$ for all $(h_k)_{k\in X} \in \ell^p_{\omega}(X;\Hil)$ \cite[Chapter 1]{HyNeVeWe16}. In particular, $g:= D_{T^d} (g_k)_{k\in X} \in \Hil^q_{1/\omega}(T,T^d)$ due to Proposition \ref{Donto}. We will show that 
\begin{equation}\label{goal1}
\rho(f) = \sum_{k\in X} \langle T^d_k f, T_k g \rangle \qquad (\forall f\in \Hil^p_{\omega}(T^d,T)), 
\end{equation}
which readily implies that $\kappa$ is surjective. By density, it suffices to verify (\ref{goal1}) for $f\in \Hil^{00}(T)$. To this end, fix some arbitrary $f\in \Hil^{00}(T)$ and let $\varepsilon >0$. By density of $\ell^1_{1/\omega}(X;\Hil)$ in $\ell^q_{1/\omega}(X;\Hil)$, there exists some $(h^{\varepsilon}_k)_{k\in X} \in \ell^1_{1/\omega}(X;\Hil)$, such that $\Vert (g_k)_{k\in X} - (h^{\varepsilon}_k)_{k\in X}\Vert_{\ell^q_{1/\omega}(X;\Hil)}<\varepsilon$. Then, again by Proposition \ref{Donto}, $h^{\varepsilon}:= D_{T^d}(h^{\varepsilon}_k)_{k\in X} \in \Hil^1_{1/\omega}(T,T^d)$. Using Hölder's inequality, we may start estimating  
\begin{flalign}\label{dualA}
&\left\vert \rho(f) - \sum_{k\in X} \langle T^d_k f, T_k g \rangle \right\vert \notag \\\
&\leq \left\vert \rho(f) - \sum_{k\in X} \langle T^d_k f, T_k h^{\varepsilon} \rangle \right\vert + \left\vert \sum_{k\in X} \langle T^d_k f, T_k (g - h^{\varepsilon})\rangle \right\vert  \notag \\
&\leq \left\vert \rho(f) - \sum_{k\in X} \langle T^d_k f, T_k h^{\varepsilon} \rangle \right\vert + \sum_{k\in X} \Vert T^d_k f \Vert \omega(k) \Vert T_k (g - h^{\varepsilon})\Vert \omega(k)^{-1} \notag \\
&\leq \left\vert \rho(f) - \sum_{k\in X} \langle T^d_k f, T_k h^{\varepsilon} \rangle \right\vert + \Vert f \Vert_{\Hil^p_{\omega}(T^d,T)}  \Vert G_{T,T^d} ((g_k)_{k\in X} - (h^{\varepsilon}_k)_{k\in X})\Vert_{\ell^q_{1/\omega}(X;\Hil)} \notag \\
&< \left\vert \rho(f) - \sum_{k\in X} \langle T^d_k f, T_k h^{\varepsilon} \rangle \right\vert + \varepsilon \Vert f \Vert_{\Hil^p_{\omega}(T^d,T)}  \Vert G_{T,T^d} \Vert_{\B(\ell^q_{1/\omega}(X;\Hil))}.
\end{flalign}
Next, let $F_1 \subseteq F_2\subseteq \dots$ be some nested sequence of finite subsets of $X$ such that $\bigcup_{n=1}^{\infty} F_n = X$. Then we have for each $k\in X$ that 
\begin{flalign}\label{XY}
\frac{1}{\omega(k)}\left\Vert T_k h^{\varepsilon} - \sum_{l\in X} T_k (T^d_l)^* h^{\varepsilon}_l \right\Vert &= \lim_{n\rightarrow \infty} \frac{1}{\omega(k)} \left\Vert T_k h^{\varepsilon} - T_k \sum_{l\in F_n} (T^d_l)^* h^{\varepsilon}_l \right\Vert  \notag \\
&\leq \lim_{n\rightarrow \infty} \left\Vert C_T \left( h^{\varepsilon} - \sum_{l\in F_n} (T^d_l)^* h^{\varepsilon}_l \right) \right\Vert_{\ell^1_{1/\omega}(X;\Hil)} \notag \\
&= \lim_{n\rightarrow \infty} \left\Vert \sum_{l\in X} (T^d_l)^* h^{\varepsilon}_l - \sum_{l\in F_n} (T^d_l)^* h^{\varepsilon}_l \right\Vert_{\Hil^1_{1/\omega}(T,T^d)} = 0 \notag
\end{flalign}
by continuity of $D_{T^d}$ and unconditional convergence of $D_{T^d}(h^{\varepsilon}_l)_{l\in X}$ (see Proposition \ref{Donto}). Thus $T_k h^{\varepsilon} =\sum_{l\in X} T_k (T^d_l)^* h^{\varepsilon}_l$ in $\Hil$ for each $k\in X$. Consequently, we may rewrite the first term in (\ref{dualA}) to
\begin{equation}\label{dualB}
\left\vert \rho(f) - \sum_{k\in X} \langle T^d_k f, T_k h^{\varepsilon} \rangle \right\vert = \left\vert \rho(f) - \sum_{k\in X} \sum_{l\in X}\langle T^d_k f, T_k (T^d_l)^*h^{\varepsilon}_l \rangle \right\vert.    
\end{equation}
Now, we interchange the order of summation in (\ref{dualB}) via Fubini's theorem. This is possible, since $(h^{\varepsilon}_l)_{l\in X} \in \ell^1_{1/\omega}(X;\Hil)$; indeed, 
\begin{flalign}
\sum_{l\in X} \sum_{k\in X}\vert \langle T^d_k f, T_k (T^d_l)^*h^{\varepsilon}_l \rangle \vert &\leq  \sum_{l\in X} \sum_{k\in X}\Vert T^d_k f\Vert \omega(l) \Vert T_k (T^d_l)^* \Vert \Vert h^{\varepsilon}_l \Vert \omega(l)^{-1} \notag \\
&\leq \sum_{l\in X} \Vert h^{\varepsilon}_l \Vert \omega(l)^{-1} \sup_{m\in X}\sum_{k\in X} \Vert T^d_m T_k^*  \Vert \Vert T^d_k f\Vert \omega(m) \notag \\
&= \Vert (h^{\varepsilon}_l)_{l\in X} \Vert_{\ell^1_{1/\omega}(X;\Hil)} \Vert G \cdot (\Vert T^d_k f\Vert)_{k\in X} \Vert_{\ell^{\infty}_{\omega}(X)} \notag \\
&\leq \Vert (h^{\varepsilon}_l)_{l\in X} \Vert_{\ell^1_{1/\omega}(X;\Hil)} \Vert G \Vert_{\B(\ell^{\infty}_{\omega}(X))} \Vert C_{T^d} f\Vert_{\ell^{\infty}_{\omega}(X)} < \infty,  \notag
\end{flalign}
where we used that the scalar matrix $G = \big[\Vert T^d_k T_l^*\Vert \big]_{k,l\in X}$ is contained in $\B(\ell^{\infty}_{\omega}(X))$ due to Lemma \ref{vectortoscalar} and since $\omega$ is $(\A,1)$-admissible. Continuing the calculation (\ref{dualB}), we see again via Hölder's inequality that 
\begin{flalign}
\left\vert \rho(f) - \sum_{k\in X} \langle T^d_k f, T_k h^{\varepsilon} \rangle \right\vert &= \left\vert \rho(f) - \sum_{l\in X}\left\langle \sum_{k\in X} T_k ^* T^d_k f, (T^d_l)^*h^{\varepsilon}_l \right\rangle \right\vert \notag \\    
&= \left\vert \rho(f) - \sum_{l\in X}\langle f, (T^d_l)^*h^{\varepsilon}_l \rangle \right\vert \notag \\    
&= \left\vert \rho(f) - \sum_{l\in X}\langle T^d_l f, h^{\varepsilon}_l \rangle \right\vert \notag \\
&\leq \left\vert \rho(f) - \sum_{l\in X}\langle T^d_l f, g_l \rangle \right\vert + \left\vert \sum_{l\in X}\langle T^d_l f, g_l - h^{\varepsilon}_l \rangle \right\vert \notag \\
&< \left\vert \rho(f) - \rho(D_T C_{T^d}f) \right\vert + \varepsilon \Vert C_{T^d} f \Vert_{\Hil^p_{\omega}(T^d,T)} \notag \\
&= \varepsilon \Vert f \Vert_{\Hil^p_{\omega}(T^d,T)}, \notag
\end{flalign}
where we also applied Theorem \ref{reconstructioncorollary} twice. Combining the latter with (\ref{dualA}) we finally arrive at 
$$\left\vert \rho(f) - \sum_{k\in X} \langle T^d_k f, T_k g \rangle \right\vert < \varepsilon \Vert f \Vert_{\Hil^p_{\omega}(T^d,T)} \left(1+  \Vert G_{T,T^d} \Vert_{\B(\ell^q_{1/\omega}(X;\Hil))}\right).$$
Since $\varepsilon$ can be arbitrarily small, (\ref{goal1}) follows. 

To show that $\kappa$ is injective, it suffices to show the existence of a left-inverse of $\kappa$. Consider the map $R_{D_T}:(\Hil^p_{\omega}(T^d,T))^{*} \longrightarrow (\ell^p_{\omega}(X;\Hil))^{*}, \rho \mapsto \rho \circ D_T$. Then it is straightforward to verify that $R_{D_T}$ is linear and bounded. Next, let $J:(\ell^p_{\omega}(X;\Hil))^{*} \longrightarrow \ell^q_{1/\omega}(X;\Hil)$ denote the canonical isomorphism satisfying $\mu(\, \cdot \, ) = \sum_{k\in X}\langle [\, \cdot \, ]_k, [J(\mu)]_k \rangle$ for $\mu \in (\ell^p_{\omega}(X;\Hil))^{*}$ \cite[Chapter 1]{HyNeVeWe16}. Then Proposition \ref{Donto} guarantees that $\Omega := D_{T^d} \circ J \circ R_{D_T} : (\Hil^p_{\omega}(T^d,T))^{*} \longrightarrow \Hil^q_{1/\omega}(T,T^d)$ is linear and bounded and we claim that $\Omega$ is a left-inverse of $\kappa$. By boundedness of $\Omega$ and $\kappa$ it suffices to verify the identity $\Omega \kappa (g) = g$ on the dense subspace $\Hil^{00}(T^d)$ of $\Hil^q_{1/\omega}(T,T^d)$. To this end, fix some arbitrary $g\in \Hil^{00}(T^d)$. Then for any $f\in \Hil^{00}(T)$ we have due to g-frame reconstruction in $\Hil^{00}(T)\subseteq \Hil$ that 
\begin{flalign}
\langle f,g \rangle &= \sum_{k\in X} \langle T^d_k f, T_k g \rangle \notag \\
&= \kappa(g) (f) \notag \\ 
&= \kappa(g) (D_T C_{T^d}f) \notag \\
&= R_{D_T}\kappa(g)(C_{T^d}f) \notag \\
&= \sum_{k\in X} \langle T^d_k f, [J R_{D_T} \kappa(g)]_k \rangle =: (\ast).\notag
\end{flalign}
On the other hand we also have $\kappa(g) \in \Hil^{*}$, which implies $R_{D_T}\kappa(g) \in (\ell^2(X;\Hil))^{*}$ and thus $J R_{D_T} \kappa(g) \in \ell^2(X;\Hil)$. Therefore, we may write
$$(\ast) = \langle C_{T^d}f, J R_{D_T} \kappa(g) \rangle_{\ell^2(X;\Hil)} = \langle f, D_{T^d} J R_{D_T} \kappa(g) \rangle_{\Hil} = \langle f, \Omega \kappa (g) \rangle_{\Hil} .$$
In other words, we have shown that $\langle f, \Omega \kappa (g) - g \rangle = 0$ for all $f\in \Hil^{00}(T)$, which, by density of $\Hil^{00}(T)$ in $\Hil$, yields $g = \Omega \kappa g \in \Hil^{00}(T)$. Thus $\kappa$ is injective.

\noindent \emph{Case (b): $p=0$.}

In the case $(p,q) = (0,1)$ the proof is analogous to case (a); we only point out the following minor differences: The sesquilinear form $\beta$ on $\Hil^0_{\omega}(T^d,T) \times \Hil^1_{1/\omega}(T,T^d)$ is bounded by $\Vert f \Vert_{\Hil^{0}_{\omega}(T^d,T)} \Vert g \Vert_{\Hil^1_{1/\omega}(T,T^d)}$; $\rho \circ D_T \in (\ell^0_{\omega}(X;\Hil))^{*}$ for every $\rho \in \Hil^0_{\omega}(T^d,T)$ (again due to Proposition \ref{Donto}); and we use that $J:(\ell_{\omega}^{0}(X;\Hil))^* \longrightarrow \ell_{1/\omega}^1(X;\Hil)$ is an isomorphism induced by the sesquilinear form $\sum_{k\in X} \langle f_k, g_k\rangle$ on $\ell_{\omega}^{0}(X;\Hil) \times \ell_{1/\omega}^1(X;\Hil)$ \cite[Theorem 3.1]{YaSr14}. Note, that the surjectivity of $\kappa$ can be shown in a more straightforward manner, because in this case $(g_k)_{k\in X}$ is already contained in $\ell^1_{1/\omega}(X;\Hil)$ and no density argument via $h^{\varepsilon}$ is required to justify the change of order of summation. 

\noindent \emph{Case (c): $p=1$.}

In the case $p=1$ (i.e. $q=\infty$), the adaptations of case (a) to this setting are as follows: Recalling that $C_{T}u = (T_k u)_{k\in X} := (\lim_{n\rightarrow \infty}T_k u_n)_{k\in X} \in \ell^{\infty}_{1/\omega}(X;\Hil)$ whenever $u = [\lbrace u_n \rbrace_{n=1}^{\infty}]_{\sim_{T,1/\omega}} \in \Hil_{1/\omega}^{\infty}(T,T^d)$, we again see via Hölder's inequality that $\beta:\Hil^1_{\omega}(T^d,T) \times \Hil^{\infty}_{1/\omega}(T,T^d) \longrightarrow \mathbb{C}$, $\beta(f,g) = \sum_{k\in X} \langle T^d_k f, T_k g \rangle$ defines a bounded sesquilinear form. So, $\kappa$ can be defined again analogously as in case (a). 

We show that $\kappa$ is surjective. As in case (a), for given $\rho \in (\Hil^1_{\omega}(T^d,T))^*$, there exists a unique $(g_k)_{k\in X} \in \ell^{\infty}_{1/\omega}(X;\Hil)$ with $\rho(D_T (h_k)_{k\in X}) = \sum_{k\in X} \langle h_k, g_k \rangle$ for all $(h_k)_{k\in X} \in \ell^1_{\omega}(X;\Hil)$. We again set $g:= D_{T^d} (g_k)_{k\in X} \in \Hil^{\infty}_{1/\omega}(T,T^d)$ and calculate via Theorem \ref{reconstructioncorollary}
\begin{flalign}
\rho(f) &= \rho(D_T C_{T^d}f) \notag \\
&= \sum_{k\in X} \langle T^d_k f, g_k \rangle \notag \\
&= \sum_{k\in X} \lim_{n\rightarrow \infty} \left\langle T^d_k \sum_{l\in F_n} T_l^* T^d_l f, g_k \right\rangle \notag \\
&= \sum_{k\in X} \lim_{n\rightarrow \infty} \sum_{l\in F_n} \left\langle T^d_l f, T_l(T^d_k)^* g_k \right\rangle \notag \\
&= \sum_{k\in X} \sum_{l\in X}\left\langle T^d_l f, T_l (T^d_k)^* g_k \right\rangle \notag \\
&= \sum_{l\in X} \sum_{k\in X}\left\langle T^d_l f, T_l (T^d_k)^* g_k \right\rangle \notag \\
&= \sum_{l\in X} \lim_{n\rightarrow \infty} \left\langle T^d_l f, T_l \sum_{k\in F_n} (T^d_k)^* g_k .\right\rangle \notag \\
&= \sum_{l\in X} \left\langle T^d_l f, T_l g \right\rangle .\notag
\end{flalign}
This time, the order of summation may be switched due to 
\begin{flalign}
\sum_{l\in X} \sum_{k\in X}\vert \langle T^d_l f, T_l (T^d_k)^* g_k \rangle \vert &\leq  \sum_{l\in X} \sum_{k\in X}\Vert T^d_l f\Vert \omega(l) \Vert T_l (T^d_k)^* \Vert \Vert g_k \Vert \omega(l)^{-1} \notag \\
&\leq \sum_{l\in X} \Vert T^d_l f\Vert \omega(l) \sup_{m\in X}\sum_{k\in X} \Vert T_m(T^d_k)^*  \Vert \Vert g_k\Vert \omega(m)^{-1} \notag \\
&= \Vert f \Vert_{\Hil^1_{\omega}(T^d,T)} \left\Vert G' \cdot (\Vert g_l \Vert)_{l\in X} \right\Vert_{\ell^{\infty}_{1/\omega}(X)} \notag \\
&\leq \Vert f \Vert_{\Hil^1_{\omega}(T^d,T)} \left\Vert G' \right\Vert_{\B(\ell^{\infty}_{1/\omega}(X))} \Vert(g_l)_{l\in X}\Vert_{\ell^{\infty}_{1/\omega}(X;\Hil)} < \infty ,\notag
\end{flalign}
where we analogously used that $G' = 
\big[ \Vert T_k(T^d_l)^*\Vert \big]_{k,l\in X} \in \B(\ell^{\infty}_{1/\omega}(X))$.

Finally, we show the injectivity of $\kappa$ by proving the existence of a left-inverse $\Omega$ of $\kappa$. As in case (a), we consider the operator $R_{D_T} \in \B((\Hil^1_{\omega}(T^d,T))^{*}, (\ell^1_{\omega}(X;\Hil))^{*})$ and the canonical isomorphism $J:(\ell^1_{\omega}(X;\Hil))^{*} \longrightarrow \ell^{\infty}_{1/\omega}(X;\Hil)$. For arbitrary $f\in \Hil^{00}(T) \subseteq \Hil^1_{\omega}(T^d,T)$ and $g = [\lbrace g_n \rbrace_{n=1}^{\infty}]_{\sim_{T,1/\omega}} \in \Hil^{\infty}_{1/\omega}(T,T^d)$ we see that 
$$\kappa(g)(f) = \sum_{k\in X}\lim_{n\rightarrow \infty} \langle T^d_k f, T_k g_n \rangle = \lim_{n\rightarrow \infty} \sum_{k\in X}\langle T^d_k f, T_k g_n\rangle = \lim_{n\rightarrow \infty} \langle f, g_n\rangle ,$$
where we applied g-frame reconstruction in $\Hil$ in the third step and the dominated convergence theorem in the second step upon noticing that $C_{T^d}f \in \ell^1_{\omega}(X;\Hil)$ and $\Vert T_k g_n \Vert \omega(k)^{-1} \leq C$ due to (\ref{Cineq}). At the same time we have as above  
\begin{flalign}
\kappa(g) (f) &= \kappa(g) (D_T C_{T^d}f) \notag \\
&= R_{D_T}\kappa(g)(C_{T^d}f) \notag \\
&= \sum_{k\in X} \langle T^d_k f, [J R_{D_T} \kappa(g)]_k \rangle \notag \\
&= \lim_{n\rightarrow \infty} \left\langle f, \sum_{k\in F_n} (T^d_k)^* [J R_{D_T} \kappa(g)]_k \right\rangle \notag
\end{flalign}
with $J R_{D_T} \kappa(g)\in \ell^{\infty}_{1/\omega}(X;\Hil)$. The composition $\Omega := D_{T^d} \circ J \circ R_{D_T}$ is an element in $\B(\Hil^1_{\omega}(T^d,T))^{*}, \Hil^{\infty}_{1/\omega}(T,T^d))$ and we have 
\begin{flalign}
\Omega \kappa (g) &= \left[ \left\lbrace \sum_{k\in F_n} (T^d_k)^* [J R_{D_T} \kappa(g)]_k \right\rbrace_{n=1}^{\infty}\right]_{\sim_{T,1/\omega}} \notag \\
&= [\lbrace (\Omega \kappa (g))_n \rbrace_{n=1}^{\infty}]_{\sim_{T,1/\omega}} \in \Hil^{\infty}_{1/\omega}(T,T^d). \notag    
\end{flalign}
Therefore, we may conclude from the above that 
$$\lim_{n\rightarrow \infty} \langle f, g_n\rangle = \kappa(g)(f) = \lim_{n\rightarrow \infty} \langle f, (\Omega \kappa (g))_n \rangle \qquad  (\forall f\in \Hil^{00}(T)).$$
Specifying to elements of the form $f = T_k^* h \in \Hil^{00}(T)$ ($k\in X, h\in \Hil$) implies
\begin{equation}\label{dualityfinal}
    \lim_{n\rightarrow \infty} \langle h, T_k \big( g_n - (\Omega \kappa (g))_n\big) \rangle = 0\qquad  (\forall k\in X, \forall h\in \Hil).
\end{equation}
However, since $g - \Omega \kappa (g) \in \Hil^{\infty}_{1/\omega}(T,T^d)$, we know (see Definition \ref{Hwinftydef}) that for each $k\in X$ the limit $\lim_{n\rightarrow \infty} T_k \big( g_n - (\Omega \kappa (g))_n\big) = T_k (g - \Omega \kappa (g)) \in \Hil$ exists and has to be zero as a consequence of (\ref{dualityfinal}). In particular, we have shown that $\lim_{n\rightarrow \infty} \Vert T_k \big( g_n - (\Omega \kappa (g))_n\big) \Vert = 0$ for each $k\in X$. Since $g - \Omega \kappa (g) \in \Hil^{\infty}_{1/\omega}(T,T^d)$ we also have $\sup_{n\in \mathbb{N}} \Vert C_{T} \big( g_n - (\Omega \kappa (g))_n\big) \Vert_{\ell^{\infty}_{1/\omega}(X;\Hil)} < \infty$. Thus $\lbrace g_n \rbrace_{n=1}^{\infty} \sim_{T,1/\omega}  \lbrace (\Omega \kappa (g))_n \rbrace_{n=1}^{\infty}$, i.e. $g = \Omega \kappa (g)$ in $\Hil^{\infty}_{1/\omega}(T,T^d)$. This completes the proof.
\end{proof}

\section{Application to Gabor g-frames}

We now turn to the Hilbert space $\Hil = L^2(\Rd)$ and consider Gabor g-frames, which were introduced in \cite{skret2020}.

We recall some basic notions from time-frequency analysis \cite{gr01}. The fundamental operators in time-frequency analysis are the \emph{translation operator} $T_x$ and the \emph{modulation operator} $M_{\omega}$ given by
$$T_x f(t) = f(t-x) \quad \text{and} \quad M_{\omega}f(t) = e^{2\pi i \omega \cdot t}f(t) \qquad (x,\omega \in \Rd).$$
Their composition $\pi(z) = M_{\omega} T_x$, where $z=(x,w) \in \mathbb{R}^{2d}$, is called a \emph{time-frequency shift} by $z$. Each of the operators $T_x$, $M_{\omega}$ and $\pi(z)$ is unitary on $L^2(\Rd)$. If $g$ is some window function, then the \emph{short-time Fourier transform} (STFT) of a function $f$ with respect to the window $g$ is given by  
$$V_g f (x,w) = \int_{\Rd} f(t) \overline{g(t-x)} e^{-2\pi i \omega \cdot t} \, dt \qquad (x,w\in \Rd).$$
Setting $z=(x,w)\in \mathbb{R}^{2d}$, we may rewrite the above to
$$V_g f (z) = \langle f, \pi(z) g \rangle \qquad (z\in \mathbb{R}^{2d})$$
whenever the bracket is well-defined. An important class of function spaces in time-frequency analysis are the \emph{modulation spaces} defined as follows. 
Let $g_0(t) = 2^{\frac{d}{4}}e^{-\pi t^2}$ denote the normalized Gaussian in $L^2(\mathbb{R}^d)$ (where $t^2 = t\cdot t$) and let $\nu_s(x) = (1+\vert x \vert)^s$ be the standard polynomial weight. Then $M^1_{\nu_s}(\Rd)$ is the Banach space of all $f \in L^2(\mathbb{R}^d)$ such that
\begin{equation}\label{M1vdef}
    \Vert \varphi \Vert_{M^1_{\nu_s}(\Rd)} =  \int_{\mathbb{R}^{2d}} \vert V_{g_0} f(z) \vert \nu_s(z) \, dz < \infty .
\end{equation}
The space $S_0 = M^1_{\nu_0}$ is known as \emph{Feichtinger's algebra} \cite{fei81}. For a $\nu_s$-moderate weight $m$ and $p\in [1,\infty]$ the \emph{modulation space} $M^p_m(\Rd)$ is the space of all elements $f$ in the anti-linear dual space $(M^1_{\nu_s}(\Rd))'$ for which 
\begin{equation}\label{Mpnorm}
    \Vert \varphi \Vert_{M^p_m(\Rd)} = \left( \int_{\mathbb{R}^{2d}} \vert V_{g_0}f(z) \vert^p m(z)^p \, dz \right)^{\frac{1}{p}} < \infty .
\end{equation}

We collect some important properties of modulation spaces we will need.

\begin{proposition}\label{Mpprop}\cite{gr01}
    Let $m$ be a $\nu_s$-moderate weight and $p\in [1,\infty]$.
\begin{itemize}
    \item[(a)] $M^p_m(\mathbb{R}^d)$ is a Banach space with respect to the norm defined in (\ref{Mpnorm}).
    \item[(b)] If $1 \leq p_1 \leq p_2 \leq \infty$ and $m_2(z) \leq Cm_1(z)$ for some $C>0$, then $M^{p_1}_{m_1}(\mathbb{R}^d)$ is continuously and densely embedded into $M^{p_2}_{m_2}(\mathbb{R}^d)$.
    \item[(c)] $M^2_1(\mathbb{R}^d) = L^2(\mathbb{R}^d)$ with equivalent norms.
\end{itemize}
\end{proposition}
Returning to the g-frame setting, we are interested in families consisting of shifts of some \emph{window operator} $T\in \mathcal{B}(L^2(\mathbb{R}^d))$ in the time-frequency plane. More precisely, the \emph{translation of the operator} $T$ by $z\in \mathbb{R}^{2d}$ \cite{koz92} is given by 
$$\alpha_z (T) := (\pi(z) \otimes \pi(z))(T) = \pi(z) T \pi(z)^*.$$
For a countable set $X \subseteq \mathbb{R}^{2d}$ and a window operator $T\in \mathcal{B}(L^2(\mathbb{R}^d))$ we call the family
$$\G = \G(T,X) = (\alpha_k (T))_{k\in X}$$
a \emph{g-Gabor system}. In the terminology of \cite{skret2020}, such a g-Gabor system $\G(T,X)$ is called a \emph{Gabor g-frame}, if it is a g-frame for $L^2(\mathbb{R}^d)$, i.e. if there exist positive constants $0<A\leq B<\infty$ such that 
\begin{equation}\label{Gaborgframe}
    A \Vert f \Vert^2 \leq \sum_{k \in X} \Vert \alpha_{k} (T) f \Vert^2 \leq B \Vert f \Vert^2 \qquad (\forall f\in L^2(\mathbb{R}^d)).
\end{equation}
In \cite{skret2020} \emph{regular} Gabor g-frames (which correspond to $X=\Lambda \subset \mathbb{R}^{2d}$ being a full-rank lattice) were considered. In particular, via Fourier methods for periodic operators, which ultimately rely on the group structure of the lattice $k$, Skrettingland proved in \cite{skret2020} several results for this class of g-frames, which typically are concluded from suitable localization properties of a given frame in some abstract Hilbert space $\Hil$ \cite{forngroech1}. We will extend some of those results to the case of \emph{irregular} Gabor g-frames (i.e. Gabor g-frames indexed by some arbitrary relatively separated set $X \subset \mathbb{R}^{2d}$) by showing that g-Gabor systems with respect to some suitably nice window operator $T$ are polynomially localized (in the sense of Example \ref{spectralexamples} (1.)).

To this end, we consider the class $\mathcal{B}_{\nu_s \otimes \nu_s}$ of integral operators acting boundedly on $L^2(\mathbb{R}^d)$, whose integral kernel belongs to $M^1_{\nu_s \otimes \nu_s}(\mathbb{R}^{2d})$, where 
\begin{flalign}
    M^1_{\nu_s \otimes \nu_s}(\mathbb{R}^{2d}) &\cong  M^1_{\nu_s}(\mathbb{R}^d) \hat{\otimes} M^1_{\nu_s}(\mathbb{R}^d) \notag \\
    &= \left\lbrace \sum_{n\in \mathbb{N}} \varphi_n \otimes \psi_n : \sum_{n\in \mathbb{N}} \Vert \varphi_n \Vert_{M^1_{\nu_s}(\mathbb{R}^d)} \Vert \psi_n \Vert_{M^1_{\nu_s}(\mathbb{R}^d)} < \infty \right\rbrace \end{flalign}
is the projective tensor product of the Banach space $M^1_{\nu_s}(\mathbb{R}^d)$ with itself.
Here $g\otimes f$ denotes the rank-one operator given by $(g\otimes f)(h) = \langle h, f\rangle g$, where $f,g,h\in L^2(\mathbb{R}^d)$.
The space $\mathcal{B}_{\nu_s \otimes \nu_s}$ was studied in \cite{skret2020} in detail. We collect the following properties, which we will need.

\begin{proposition}\label{beautifuloperatorsprop}
\cite{skret2020} Let $T\in \mathcal{B}_{\nu_s \otimes \nu_s}$. Then the following hold:
\begin{itemize}
    \item[(a)] There exist sequences $\lbrace \varphi_n\rbrace_{n\in \mathbb{N}}$ and $\lbrace \psi_n\rbrace_{n\in \mathbb{N}}$ in $M^1_{\nu_s}(\mathbb{R}^d)$ with 
    $$\sum_{n\in \mathbb{N}} \Vert \varphi_n \Vert_{M^1_{\nu_s}(\mathbb{R}^d)} \Vert \psi_n \Vert_{M^1_{\nu_s}(\mathbb{R}^d)} < \infty,$$
    such that $T$ can be written as a sum of rank-one operators
    \begin{equation}\label{rankonesum}
        T = \sum_{n\in \mathbb{N}} \varphi_n \otimes \psi_n.
    \end{equation}
    The decomposition (\ref{rankonesum}) converges absolutely in $\mathcal{B}_{\nu_s \otimes \nu_s}$, in the space of trace class operators on $L^2(\mathbb{R}^d)$, and in $\mathcal{B}(L^2(\mathbb{R}^d))$.
    \item[(b)] Let $T^* \in \mathcal{B}(L^2(\mathbb{R}^d))$ denote the adjoint operator of $T$. Then $T^* \in \mathcal{B}_{\nu_s \otimes \nu_s}$. In particular, if $T = \sum_{n\in \mathbb{N}} \varphi_n \otimes \psi_n$, then $T^* = \sum_{n\in \mathbb{N}} \psi_n \otimes \varphi_n$.
\end{itemize}
\end{proposition}

We are now ready to prove the main theorem of this section, which provides a concrete example of an intrinsically localized operator-valued sequence.

\begin{theorem}\label{polynomiallylocalizedgframe}
Let $X\subset \mathbb{R}^{2d}$ be relatively separated, $s>2d$, and $T\in \mathcal{B}_{\nu_s \otimes \nu_s}$. Then the g-Gabor system $\G=\G(T,X)$ is intrinsically $\J_s$-localized. In particular, $\G(T,X)$ is a g-Bessel sequence for $L^2(\Rd)$.  
\end{theorem}

\begin{proof}
If $T\in \mathcal{B}_{\nu_s \otimes \nu_s}$, then, by Proposition \ref{beautifuloperatorsprop}, also $T^*\in \mathcal{B}_{\nu_s \otimes \nu_s}$ and there exist sequences $\lbrace \varphi_n\rbrace_{n\in \mathbb{N}}, \lbrace \psi_n\rbrace_{n\in \mathbb{N}}$ in $M^1_{\nu_s}(\mathbb{R}^d)$ with $\sum_{n\in \mathbb{N}} \Vert \varphi_n \Vert_{M^1_{\nu_s}} \Vert \psi_n \Vert_{M^1_{\nu_s}} < \infty$ such that $T = \sum_{n\in \mathbb{N}} \varphi_n \otimes \psi_n$ and $T^* = \sum_{n\in \mathbb{N}} \psi_n \otimes \varphi_n$. We first show that there exists some constant $C>0$ such that
$$\Vert [G_{\G}]_{k,l} \Vert \leq C (1+\vert k - l \vert)^{-s} \qquad (\forall k, l \in X).$$
Observe that $[G_{\G}]_{k, l} = \alpha_{k}(T) \alpha_{l}(T)^* = \pi(k) T\pi(k)^* \pi(l) T^* \pi(l)^*$. Then for arbitrary $f\in L^2(\mathbb{R}^d)$ we have
\begin{flalign}
    [G_{\G}]_{k, l} f &=  \pi(k) \left( \sum_{m\in \mathbb{N}} \varphi_m \otimes \psi_m \right) \pi(k)^* \pi(l) \left(\sum_{n\in \mathbb{N}} \psi_n \otimes \varphi_n \right) \pi(l)^* f \notag \\
    &= \left( \sum_{m\in \mathbb{N}} (\pi(k)\varphi_m) \otimes (\pi(k) \psi_m) \right) \left(\sum_{n\in \mathbb{N}} (\pi(l)\psi_n) \otimes (\pi(l)\varphi_n) \right) f \notag \\
    &= \left( \sum_{m\in \mathbb{N}} (\pi(k)\varphi_m) \otimes (\pi(k) \psi_m) \right) \sum_{n\in \mathbb{N}} \langle f , \pi(l)\varphi_n \rangle \pi(l)\psi_n \notag \\
    &= \sum_{m\in \mathbb{N}} \sum_{n\in \mathbb{N}} \langle f , \pi(l)\varphi_n \rangle \langle  \pi(l) \psi_n , \pi(k) \psi_m) \rangle \pi(k) \varphi_m .\notag
\end{flalign}
This implies that 
\begin{flalign}
    \Vert [G_{\G}]_{k, l} f \Vert_{L^2} &\leq \sum_{m\in \mathbb{N}} \sum_{n\in \mathbb{N}} \Vert f \Vert_{L^2} \Vert \varphi_n \Vert_{L^2} \Vert \varphi_m \Vert_{L^2} \vert V_{\psi_m} (\pi(l)\psi_n) (k) \vert, \notag \\
    &\leq C_0^2\sum_{m\in \mathbb{N}} \sum_{n\in \mathbb{N}} \Vert f \Vert_{L^2} \Vert \varphi_n \Vert_{M^1_{\nu_s}} \Vert \varphi_m \Vert_{M^1_{\nu_s}} \vert V_{\psi_m} \psi_n (k-l) \vert, \notag
\end{flalign}
where, in the latter estimate, we used the covariance property of the STFT \cite{gr01} and the continuous embedding $M^1_{\nu_s}(\mathbb{R}^d) \hookrightarrow L^2(\mathbb{R}^d)$ (see Proposition \ref{Mpprop}).
Since $\Vert g_0\Vert_{L^2(\mathbb{R}^d)}=1$, the STFT satisfies the pointwise estimate 
$$\vert V_{\psi_m} \psi_n \vert \leq \frac{1}{\Vert g_0 \Vert_{L^2(\mathbb{R}^d)}^2} \vert V_{g_0} \psi_n \vert \ast \vert V_{\psi_m} g_0 \vert = \vert V_{g_0} \psi_n \vert \ast \vert V_{\psi_m} g_0 \vert$$ 
(see \cite[Theorem 11.3.7]{gr01}). Hence, we may apply Young's inequality for weighted $L^p$-spaces \cite[Proposition 11.1.3]{gr01} and obtain that
$$\vert V_{\psi_m} \psi_n (k - l) \vert \nu_s(k - l) \leq \Vert V_{\psi_m} \psi_n \Vert_{L^{\infty}_{\nu_s}} \leq \Vert V_{g_0} \psi_n \Vert_{L^{\infty}_{\nu_s}} \, \Vert V_{\psi_m} g_0 \Vert_{L^1_{\nu_s}}.$$
Since $\nu_s$ is a symmetric weight it holds $\Vert V_{\psi_m} g_0 \Vert_{L^1_{\nu_s}} = \Vert V_{g_0} \psi_m \Vert_{L^1_{\nu_s}}$. Together with the continuous embedding $M^1_{\nu_s}(\mathbb{R}^d) \hookrightarrow M^{\infty}_{\nu_s}(\mathbb{R}^d)$ this yields
\begin{flalign}
  \vert V_{\psi_m} \psi_n (k - l) \vert \nu_s(k - l) 
  &\leq \Vert V_{g_0} \psi_n \Vert_{L^{\infty}_{\nu_s}} \, \Vert V_{g_0} \psi_m\Vert_{L^1_{\nu_s}} \notag \\ 
  &= \Vert \psi_n \Vert_{M^{\infty}_{\nu_s}} \, \Vert \psi_m \Vert_{M^1_{\nu_s}} \notag \\
  &\leq C_1 \Vert \psi_n \Vert_{M^1_{\nu_s}} \, \Vert \psi_m \Vert_{M^1_{\nu_s}} . \notag
\end{flalign}
Altogether, we see that
\begin{flalign}
    \Vert [G_{\G}]_{k, l} \Vert &\leq C_0^2\sum_{m\in \mathbb{N}} \sum_{n\in \mathbb{N}} \Vert \varphi_n \Vert_{M^1_{\nu_s}} \Vert \varphi_m \Vert_{M^1_{\nu_s}} \vert V_{\psi_m} \psi_n (k - l) \vert \notag \\
    &\leq \frac{C_0^2 C_1}{\nu_s(k-l)} \sum_{m\in \mathbb{N}} \sum_{n\in \mathbb{N}} \Vert \varphi_n \Vert_{M^1_{\nu_s}} \Vert \varphi_m \Vert_{M^1_{\nu_s}} \Vert \psi_n \Vert_{M^1_{\nu_s}} \Vert \psi_m \Vert_{M^1_{\nu_s}}  \notag \\
    &= \frac{C_0^2 C_1}{\nu_s(k-l)} \left(\sum_{m\in \mathbb{N}} \Vert \varphi_m \Vert_{M^1_{\nu_s}} \Vert \psi_m \Vert_{M^1_{\nu_s}} \right) \left( \sum_{n\in \mathbb{N}}  \Vert \varphi_n \Vert_{M^1_{\nu_s}}  \Vert \psi_n \Vert_{M^1_{\nu_s}} \right) \notag \\
    &= C(1+\vert k - l \vert)^{-s} . \notag
\end{flalign}
Since $k,l\in X$ were arbitrary, the latter implies that $\G(T,X)$ is intrinsically $\J_s$-localized. By Proposition \ref{locgBessel}, $\G(T,X)$ is therefore a g-Bessel sequence.
\end{proof}

In particular, if a Gabor g-system as above is a g-frame for $L^2(\mathbb{R}^d)$, we may apply our machinery established so far and obtain further results. We note that sufficiency conditions for such Gabor g-systems over a lattice being a frame were given in \cite{skret2020}. 

\begin{theorem}\label{maingaborg}
Let $X\subset \mathbb{R}^{2d}$ be relatively separated, $s>2d+r$ for some $r\geq 0$, $T\in \mathcal{B}_{\nu_s \otimes \nu_s}$, and $m$ be a $\nu_r$-moderate weight. If $\G=\G(T,X)$ constitutes a Gabor g-frame for $L^2(\Rd)$, then $\G$ and its canonical dual $\widetilde{\G}$ are intrinsically and mutually $\J_s$-localized. In other words, there exist constants $C_1, C_2, C_3 >0$ such that
\begin{flalign}\label{3ineqs}
   \Vert T\pi(k)^* \pi(l) T^* \Vert &\leq C_1 (1+\vert k - l \vert)^{-s} \qquad (\forall k, l \in X) \notag \\
   \Vert T\pi(k)^* S_{\G}^{-1} \pi(l) T^* \Vert &\leq C_2 (1+\vert k - l \vert)^{-s} \qquad (\forall k, l \in X) \notag \\
   \Vert T\pi(k)^* S_{\G}^{-2} \pi(l) T^* \Vert &\leq C_3 (1+\vert k - l \vert)^{-s} \qquad (\forall k, l \in X).
\end{flalign}
Furthermore, for each $p\in (\frac{2d}{s-r}, \infty]\cup \lbrace 0\rbrace$, the space $\Hil_m^{p}(\widetilde{\G},\G)$ is a well-defined (quasi-) Banach space and it holds
$$f = \sum_{k\in X} \pi(k)T^*T \pi(k)^* S_{\G}^{-1} f = \sum_{k\in X} S_{\G}^{-1} \pi(k)T^*T \pi(k)^* f \qquad (\forall f\in \Hil_{m}^p(\widetilde{\G},\G))$$ 
with unconditional convergence in $\Hil_m^p(\widetilde{\G},\G)$ for $p\in (\frac{2d}{s-r}, \infty)\cup \lbrace 0\rbrace$ and unconditional convergence in the $\sigma(\Hil, \Hil^{00}(\widetilde{\G}))$-topology in the case $p=\infty$ respectively.
\end{theorem}

\begin{proof}
Since $s>2d$, $\J_s$ is a spectral algebra by Example \ref{spectralexamples} (1.). Furthermore, since $T\in \mathcal{B}_{\nu_s \otimes \nu_s}$, Theorem \ref{polynomiallylocalizedgframe} guarantees that $\G$ is intrinsically $\J_s$-localized, which is equivalent to the first inequality in (\ref{3ineqs}). If $\widetilde{\G}$ denotes the canonical dual g-frame of $\G$, then Theorem \ref{duallocalized} implies that $\widetilde{\G} \sim_{\J_s} \G$ and $\widetilde{\G} \sim_{\J_s} \widetilde{\G}$, which is equivalent to the second and third inequality in (\ref{3ineqs}) respectively. Moreover, the assumptions on $m$ and Example \ref{admissibleex} (1.) imply that $m$ is an $(\J_s, \frac{2d}{s-r})$-admissible weight. Hence, an application of Corollary \ref{canonicalduallocalized} yields the remaining part of the theorem.   
\end{proof}

\begin{remark}
We expect that the modulation spaces $M_{\nu_s \otimes \nu_s}^p$ considered in \cite{skret2020} coincide (up to norm-equivalence) with the co-orbit spaces $\Hil_{\nu_s}^p(\widetilde{\G},\G)$ appearing in the latter theorem. Similarly, we expect that the co-orbit spaces of operators discussed in \cite{dölumnsk23} can be identified with co-orbit spaces associated with certain localized g-frames. We leave a more detailed investigation on this matter to future works.
\end{remark}

\section*{Acknowledgments}

The authors thank Daniel Freeman, Karlheinz Gröchenig, Jakob Holböck, Ilya Krishtal, Franz Luef and Henry McNulty  for their valuable suggestions and related discussions. 

This research was funded in whole or in part by the Austrian Science Fund (FWF) 10.55776/P34624. For open access purposes, the author has applied a CC BY public copyright license to any author accepted manuscript version arising from this submission.

\bibliographystyle{abbrv}

\bibliography{biblioall}
\end{document}